\documentclass{amsart}

\usepackage{amsmath,amsfonts,amssymb,amsthm} %AMS packages
\usepackage[all, cmtip]{xy} %For commutative diagrams
\usepackage{enumitem} %Adds extra options to enumerate environment
\usepackage{pict2e} %Extends picture environment to allow more line slopes, circle radii, etc.
\usepackage{caption} %Gives better caption formatting and more caption formatting options
\usepackage[T1]{fontenc} %Extra font characters

\usepackage{mathtools}

\theoremstyle{definition}
\newtheorem{defn}{Definition}[section]

\newtheorem{remark}[defn]{Remark}

\theoremstyle{plain}
\newtheorem{lemma}[defn]{Lemma}
\newtheorem{theorem}[defn]{Theorem}
\newtheorem{proposition}[defn]{Proposition}
\newtheorem{corollary}[defn]{Corollary}

\newcommand{\CC}{\mathbb{C}}

\newcommand{\HH}{\mathbb{H}}

\newcommand{\PP}{\mathbb{P}}
\newcommand{\QQ}{\mathbb{Q}}
\newcommand{\RR}{\mathbb{R}}

\newcommand{\VV}{\mathbb{V}}

\newcommand{\ZZ}{\mathbb{Z}}

\newcommand{\calA}{\mathcal{A}}

\newcommand{\calD}{\mathcal{D}}

\newcommand{\calF}{\mathcal{F}}

\newcommand{\calH}{\mathcal{H}}

\newcommand{\calL}{\mathcal{L}}
\newcommand{\calM}{\mathcal{M}}
\newcommand{\calN}{\mathcal{N}}
\newcommand{\calO}{\mathcal{O}}

\newcommand{\calS}{\mathcal{S}}
\newcommand{\calT}{\mathcal{T}}

\newcommand{\calX}{\mathcal{X}}

\DeclareMathOperator{\NS}{NS}
\DeclareMathOperator{\Pic}{Pic}
\DeclareMathOperator{\rank}{rank}
\DeclareMathOperator{\Aut}{Aut}

\numberwithin{table}{section}

\begin{document}

\title[Calabi-Yau Threefolds fibred by Lattice Polarized K3 Surfaces]{Calabi-Yau threefolds fibred by High Rank Lattice Polarized K3 Surfaces}

\author[C. F. Doran]{Charles F. Doran}
\address{Department of Mathematical and Statistical Sciences, 632 CAB, University of Alberta,  Edmonton, AB, T6G 2G1, Canada}
\email{charles.doran@ualberta.ca}
\thanks{C. F. Doran (University of Alberta) was supported by the Natural Sciences and Engineering Research Council of Canada, the Pacific Institute for the Mathematical Sciences, and the Visiting Campobassi Professorship at the University of Maryland.}

\author[A. Harder]{Andrew Harder}
\address{Department of Mathematics, Christmas-Saucon Hall, 14 E. Packer Ave, Bethlehem, PA, 18015, USA}
\email{anh318@lehigh.edu}
\thanks{A. Harder (Lehigh University) was partially supported by the Simons Collaboration Grant in \emph{Homological Mirror Symmetry}.}

\author[A.Y. Novoseltsev]{Andrey Y. Novoseltsev}
\address{Department of Mathematical and Statistical Sciences, 632 CAB, University of Alberta,  Edmonton, AB, T6G 2G1, Canada}
\email{novoselt@ualberta.ca}
\thanks{A. Y. Novoseltsev (University of Alberta) was supported by the University of Alberta Teaching and Learning Enhancement Fund.}

\author[A. Thompson]{Alan Thompson}
\address{Department of Mathematical Sciences, Loughborough University, Loughborough, LE11 3TU, United Kingdom}
\email{a.m.thompson@lboro.ac.uk}
\thanks{A. Thompson (Loughborough University) was supported by the Engineering and Physical Sciences Research Council programme grant \emph{Classification, Computation, and Construction: New Methods in Geometry}.}

\date{}

\begin{abstract}
We study threefolds fibred by K3 surfaces admitting a lattice polarization by a certain class of rank 19 lattices. We begin by showing that any family of such K3 surfaces is completely determined by a map from the base of the family to the appropriate K3 moduli space, which we call the generalized functional invariant. Then we show that if the threefold total space is a smooth Calabi-Yau, there are only finitely many possibilities for the polarizing lattice and the form of the generalized functional invariant. Finally, we construct explicit examples of Calabi-Yau threefolds realizing each case and compute their Hodge numbers. 
\end{abstract}
\maketitle

\section{Introduction}

The primary aim of this paper is to study threefolds fibred by K3 surfaces polarized by a certain class of rank $19$ lattice, with a particular focus on Calabi-Yau threefolds.

In more detail, this paper is concerned with the study of threefolds fibred by $M_n$-polarized K3 surfaces, where $M_n$ is the rank $19$ lattice
\[M_n := H \oplus E_8 \oplus E_8 \oplus \langle -2n \rangle;\]
here $H$ denotes the hyperbolic plane lattice and $E_8$ denotes the negative definite $E_8$ root lattice. Such K3 surfaces, and the threefolds fibred by them, are interesting from the perspective of mirror symmetry. In particular, $M_n$-polarized K3 surfaces are mirror (here ``mirror'' is used in the sense of Dolgachev \cite{mslpk3s} and Nikukin \cite{fagkk3s}) to K3 surfaces of degree $2n$, and we expect mirror symmetry for Calabi-Yau threefolds fibred by $M_n$-polarized K3 surfaces to be closely linked to the Fano-LG correspondence for smooth Fano threefolds of Picard rank $1$ (see Remarks \ref{rem:mirrors} and \ref{rem:mirrors2}).

Our study of threefolds fibred by $M_n$-polarized K3 surfaces was initiated in \cite{cytfmqk3s}, where we performed a careful examination of threefolds fibred by $M_2$-polarized K3 surfaces. In particular, we gave a classification of Calabi-Yau threefolds fibred non-isotrivially by $M_2$-polarized K3 surfaces. This classification raised a natural question: for which $n$ do there exist Calabi-Yau threefolds fibred non-isotrivially by $M_n$-polarized K3 surfaces, and can they also be classified? The aim of this paper is to answer that question.

Our first main result (Theorem \ref{genfuninv}) is a generalization of \cite[Theorem 2.3]{cytfmqk3s}. It shows that, for $n \geq 2$, an $M_n$-polarized family of K3 surfaces (in the sense of \cite[Definition 2.1]{flpk3sm}) over a quasi-projective base curve $U$ is completely determined by its \emph{generalized functional invariant map} $U \to \calM_{M_n}$, where $\calM_{M_n}$ denotes the moduli space of $M_n$-polarized K3 surfaces. The case $n = 1$ is excluded here, as it is complicated by the fact that $M_1$-polarized K3 surfaces admit an antisymplectic involution that fixes the $M_1$-polarization. This means that the analogue of Theorem  \ref{genfuninv} does not hold for $M_1$-polarized families of K3 surfaces (see Remark \ref{rem:M1}); we therefore primarily restrict our attention to the cases where $n \geq 2$. It should, however, be noted that several Calabi-Yau threefolds fibred by $M_1$-polarized K3 surfaces are known to exist (see \cite[Theorem 5.20]{flpk3sm}), but no classification for them is currently known; the authors intend to address this in future work.

Following on from this, the main results of the paper (Theorems \ref{thm:H1CY3} and \ref{thm:rigid}) show that a Calabi-Yau threefold may only admit a non-isotrivial fibration by $M_n$-polarized K3 surfaces if 
\[n \in \{1,2,3,4,5,6,7,8,9,10,11,12,13,14,15,16,17,19,23\}\] 
and, moreover, that the resulting Calabi-Yau threefolds are rigid unless 
\[n \in \{1,2,3,4,5,6,7,8,9,11\}.\]
In addition to these, Theorem \ref{thm:partitions} places strict limits on the possible generalized functional invariant maps that may arise for each $n$.

In the final part of this paper we study the converse statement, and ask which of the cases allowed by the theorems mentioned above can actually be realized by examples. Our main result here (Theorem \ref{thm:CY3}) constructs explicit birational models for all such Calabi-Yau threefolds, and we compute Hodge numbers in all cases. This gives a substantial extension of the results of \cite{cytfmqk3s}.

\subsection{Structure of the paper}

This paper is structured as follows. In Section \ref{sec:latticebound} we begin by proving Theorem \ref{genfuninv}, which shows that any non-isotrivial $M_n$-polarized family of K3 surfaces, for $n \geq 2$, is uniquely determined by its generalized functional invariant map. Then we prove a Hodge-theoretic version of the same result: that the transcendental variation of Hodge structure associated to an $M_n$-polarized family of K3 surfaces is the pull-back of a fundamental variation of Hodge structure $\VV_n^+$ from the moduli space $\calM_{M_n}$ of $M_n$-polarized K3 surfaces. Following on from this, we construct the variations of Hodge structure $\VV_n^+$ explicitly and study some of their properties. This allows us to prove Theorems \ref{thm:H1CY3} and \ref{thm:rigid}, which show that a Calabi-Yau threefold may only admit a non-isotrivial $M_n$-polarized K3 fibration if $n$ is one of a small number of possibilities. Finally, we use results from \cite{hnpfe} to prove Theorem \ref{thm:partitions}, which places severe restrictions on the possible generalized functional invariants that may give rise to Calabi-Yau threefolds.

In Section \ref{sec:modularfamilies} we study the converse problem, and attempt to construct explicit examples of Calabi-Yau threefolds realizing the possibilities allowed by our earlier results. By  Theorem \ref{genfuninv}, these are all birational to pull-backs of certain fundamental families $\calX_n \to \calM_{M_n}$ under the generalized functional invariant map. Section \ref{sec:modularfamilies} is dedicated to a detailed study of these fundamental families $\calX_n$; in particular, we conduct a careful examination of their singular fibres and resolutions.

Finally, Section \ref{sec:constCY3} applies the results of the previous sections to the explicit construction of Calabi-Yau threefolds. The main result here is Theorem \ref{thm:CY3}, which shows that all of the generalized functional invariants listed in Theorem \ref{thm:partitions} actually give rise to (possibly mildly singular) Calabi-Yau threefolds fibred by $M_n$-polarized K3 surfaces. Following on from this, Propositions \ref{prop:h11} and \ref{prop:h21} explicitly compute the Hodge numbers of these Calabi-Yau threefolds.

\begin{remark} Throughout this paper, a \emph{Calabi-Yau threefold} will always be a projective threefold $\calX$ with trivial canonical sheaf $\omega_{\calX} \cong \calO_{\calX}$ and $H^1(\calX,\calO_{\calX}) = 0$. Unless otherwise specified, we do not necessarily assume that $\calX$ is smooth, but we do assume that any singularities it has are $\QQ$-factorial and terminal.
\end{remark} 

\noindent \textbf{Acknowlegements.} The computations of singularity types and explicit resolutions of singularities found in this paper were greatly assisted by the \textsc{Singular} computer algebra system \cite{singular}.

\section{A bound on possible lattices}\label{sec:latticebound}

The aim of this section is to place restrictions on the possible lattices $M_n$ that may polarize the fibres in a Calabi-Yau threefold fibred by $M_n$-polarized K3 surfaces. We will begin by establishing some general results that apply to any threefold fibred by $M_n$-polarized K3 surfaces, then specialize to the Calabi-Yau case in Section \ref{sec:CYbound}.

We begin by setting up some notation. Let $\calX$ be a projective threefold with at worst terminal singularities, that admits a fibration $\pi\colon \calX \to B$ by K3 surfaces over a smooth base curve $B$. Let $X_p$ denote the fibre of $\pi$ over $p \in B$ and let $\NS(X_p)$ denote the N\'{e}ron-Severi group of $X_p$. We will assume that, for general choice of $p$, we have $\NS(X_p) \cong M_n$, where $M_n$ denotes the rank $19$ lattice $M_n := H \oplus E_8 \oplus E_8 \oplus \langle -2n \rangle$ (here $H$ denotes the hyperbolic plane lattice and $E_8$ denotes the negative definite $E_8$ root lattice). Let $M_n^{\perp} = H \oplus \langle 2n \rangle$ denote the orthogonal complement of $M_n$ in the K3 lattice $\Lambda_{\mathrm{K3}} := H^{\oplus 3} \oplus E_8^{\oplus 2}$.

As any singularities of $\calX$ are terminal, they must be isolated points by \cite[Corollaries 5.38 and 5.39]{bgav}. Let $U \subset B$ denote the open set over which $\calX$ is smooth and the fibres of $\pi$ are smooth K3 surfaces, let $j\colon U \hookrightarrow B$ denote the natural embedding, and let $\pi^U\colon \calX^U \to U$ denote the restriction of $\calX$ to $U$. Suppose further that $\pi^U\colon \calX^U \to U$ is an $M_n$-polarized family of K3 surfaces in the sense of the following definition.

\begin{defn}\label{defn:L-pol} \textup{\cite[Definition 2.1]{flpk3sm}}
Let $L \subseteq \Lambda_{\mathrm{K3}}$ be a lattice and $\pi^U\colon \mathcal{X}^U \rightarrow U$ be a smooth projective family of K3 surfaces over a smooth quasiprojective base $U$. We say that $\mathcal{X}^U$ is an \emph{$L$-polarized family of K3 surfaces} if
\begin{itemize}
\item there is a trivial local subsystem $\calL$ of $R^2\pi_*\mathbb{Z}$ so that, for each $p \in \Delta^*$, the fibre $\calL_p \subset H^2(X_p,\mathbb{Z})$ of $\calL$ over $p$ is a primitive sublattice of $\NS(X_p)$ that is isomorphic to $L$, and
\item there is a line bundle $\calA$ on $\calX^U$ whose restriction $\calA_p$ to any fibre $X_p$ is ample with first Chern class $c_1(\calA_p)$ contained in $\calL_p$ and primitive in $\NS(X_p)$.
\end{itemize}
\end{defn}

In general, we will call such $\calX$ a \emph{threefold fibred by $M_n$-polarized K3 surfaces}. If, in addition, the family $\pi^U\colon \calX^U \to U$ is not isotrivial, we will call  $\calX$ a \emph{threefold fibred non-isotrivially by $M_n$-polarized K3 surfaces}.

\subsection{Modular families and the generalized functional invariant} \label{sec:modularinvariants}

Results of Dolgachev \cite{mslpk3s} show that the moduli space of $M_n$-polarized K3 surfaces is isomorphic to a dense open set in a certain modular curve. This modular curve admits a standard compactification, obtained by adding finitely many points (called \emph{cusps}), which in turn gives rise to a compactification of the moduli space of $M_n$-polarized K3 surfaces. We henceforth denote this compactified moduli space by $\calM_{M_n}$; full details of its construction are given in Section \ref{sec:modularcurves}, below.

To any $M_n$-polarized family of K3 surfaces  $\pi^U\colon \mathcal{X}^U \rightarrow U$, we may associate a \emph{generalized functional invariant map} $g \colon U \to \calM_{M_n}$, defined to be the map which takes a point $p \in U$ to the point in moduli corresponding to the fibre $X_p$. The generalized functional invariant map is important because of the following theorem.

\begin{theorem} \label{genfuninv} Suppose $n \geq 2$. Let $\pi^U\colon \calX^U \to U$ denote a non-isotrivial $M_n$-polarized family of K3 surfaces over a quasi-projective curve $U$, such that the N\'{e}ron-Severi group of a general fibre of $\calX^U$ is isomorphic to $M_n$. Then $\pi^U\colon \calX^U \to U$ is uniquely determined  \textup{(}up to isomorphism\textup{)} by its generalized functional invariant map $g \colon U \to \calM_{M_n}$.
\end{theorem}
\begin{proof} This was proved for $n = 2$ in \cite[Theorem 2.3]{cytfmqk3s}. The result in the general case follows by the same argument, after noting that the discriminant group $A_{M_n^{\perp}}$ is isomorphic to $\ZZ/2n\ZZ$.\end{proof}

\begin{remark} \label{rem:M1} In the case $n=1$, the existence of an antisymplectic involution fixing the $M_1$-polarization on an $M_1$-polarized K3 surface gives rise to several pathologies, one of which is that this result does not hold. For this reason, in this paper we will always restrict our attention to the cases where $n \geq 2$, unless otherwise stated. More details about the case $n=1$ may be found in \cite[Section 5.5]{flpk3sm}; it is our intention to address this case in more detail in future work.
\end{remark}

A consequence of this theorem is that any $M_n$-polarized family of K3 surfaces $\pi^U\colon \calX^U \to U$, for $n \geq 2$, is the pull-back of a fundamental modular family $\calX_n \to \calM_{M_n}$ under the generalized functional invariant map. 

These modular families $\calX_n$ may be constructed as follows. Clingher and Doran \cite{milpk3s} proved that there is a family of K3 surfaces written as resolutions of singular quartics in $\mathbb{P}^3$
\[X(a,b,d):= \{y^2 zw - 4x^3z + 3axzw^2 - \tfrac{1}{2}(dz^2w^2 + w^4) + bzw^3 = 0\},\]
with parameters $a,b,d \in \CC$. For two sets of parameters $(a,b,d)$ and $(a',b',d')$, the corresponding K3 surfaces are isomorphic via a projective transformation if and only if $(a,b,d)$ is equal to $(\lambda^2 a',\lambda^3 b',\lambda^6 d')$ for some $\lambda \in \CC$, or in other words, if $(a,b,d)$ and $(a',b',d')$ correspond to the same point in the weighted projective surface $\mathbb{WP}(2,3,6) \cong \mathbb{P}^2$. Furthermore, every K3 surface in the above form admits a polarization by the lattice $M = E_8 \oplus E_8 \oplus H$, and every $M$-polarized K3 surface can be expressed in the above form. In other words, the family $X(a,b,d)$ defines a modular family of $M$-polarized K3 surfaces. The parameters $(a,b,d)$ may be expressed in terms of modular functions on the quasiprojective variety $\mathbb{H}^2 / (\mathrm{PSL}_2(\mathbb{Z})^2 \rtimes \mathbb{Z}/2\ZZ)$.

In particular, if $X$ is a K3 surface which is lattice polarized by the lattice $M_n = M \oplus \langle -2n \rangle$, then $X$ can be expressed as a resolution of one of the surfaces $X(a,b,d)$. One may ask whether the subvarieties of the $(a,b,d)$ domain corresponding to $M_n$-polarized K3 surfaces are indeed the moduli spaces $\mathcal{M}_{M_n}$. There is a map from $\mathcal{M}_{M_n}$ to $\mathbb{WP}(2,3,6)$, whose degree is given by the index of the subgroup of $\mathrm{O}(\Lambda_{\mathrm{K3}})$ which acts as the identity on $M_n$ inside the subgroup of $\mathrm{O}(\Lambda_{\mathrm{K3}})$ which acts as the identity on $M$ and sends $M_n^\perp$ to itself. The following result implies that this index is $1$.

\begin{lemma}\label{lemma:lat}
Suppose $n \geq 2$. If  $\gamma \in \mathrm{O}(\Lambda_{\mathrm{K3}})$ acts as the identity on the lattice $M$ and sends $M_n^\perp$ to itself, then $\gamma$ acts as the identity on $M_n$.
\end{lemma}
\begin{proof}This is a simple exercise in lattice theory using \cite[Corollary 1.5.2]{isbfa}.\end{proof}

Therefore, the curves in $\mathbb{WP}(2,3,6)$ which support $M_n$-polarized K3 surfaces may be identified with the curves $\mathcal{M}_{M_n}$ whenever $n \geq 2$. So for $n \geq 2$ we may define a family $\calX_n$ of K3 surfaces over $\calM_{M_n}$ to be a maximal crepant resolution of the pull-back of the family $X(a,b,d)$ by the map $\calM_{M_n} \to \mathbb{WP}(2,3,6)$ (note that the general fibre of $\calX_n$ is smooth, but the total space may still be singular). By construction, the fibres of $\calX_n$ are smooth K3 surfaces admitting $M_n$-polarizations and the period map is of degree $1$ onto the moduli space of $M_n$-polarized K3 surfaces. The fact that $\calX_n$ defines an $M_n$-polarized family, in the sense of Definition \ref{defn:L-pol}, then follows from Lemma \ref{lemma:lat}.

We will compute some of the families $\calX_n$ explicitly, for small values of $n$, in Section \ref{sec:modularfamilies}; but in general this is difficult to do. Therefore, in order to deduce information about $M_n$-polarized families of K3 surfaces $\pi^U\colon \calX^U \to U$ for arbitrary $n \geq 2$, we instead turn to Hodge theory.

\subsection{Some Hodge theory} The family $\pi^U\colon \calX^U \to U$ determines a weight $2$ variation of polarized Hodge structure $\calH_{\calX}$ on $R^2\pi^U_* \mathbb{Z}$. Following the discussion of \cite[Section 2.1]{flpk3sm}, define $\calT(\calX^U)$ to be the integral local subsystem of  $R^2\pi^U_* \mathbb{Z}$ underlying the smallest integral sub-variation of Hodge structure of $\calH_{\calX}$ containing $\calH^{2,0}_{\calX} := \calF^2(R^2\pi^U_* \mathbb{Z})$. The local system $\calT(\calX^U)$ thus supports a variation of polarized Hodge structure, with polarization induced from $\calH_{\calX}$ (this variation of Hodge structure is the ``essential part'' of $\calH_{\calX}$, as defined by Saito and Zucker \cite[Section 4]{cnrfk3sftat}).

Let $\calN\calS(\calX^U)$ denote the orthogonal complement of $\calT(\calX^U)$ in $R^2\pi^U_* \mathbb{Z}$. Then we have a splitting of $R^2\pi^U_*\QQ$ into a direct sum of two irreducible $\QQ$-local systems
\begin{equation} \label{eq:splitting}R^2\pi^U_*\QQ = \left(\calN\calS(\calX^U) \oplus \calT(\calX^U)\right) \otimes \QQ.\end{equation}
Here $\calN\calS(\calX^U)$ can be interpreted as those classes which lie in $\NS(X_p)$ for all $p \in U$; by assumption, in our setting $\calN\calS(\calX^U)$ is a trivial local system of rank $19$. Therefore, all of the important information is contained in the local system $\calT(\calX^U)$, which has rank $3$ and supports a variation of Hodge structure of type $(1,1,1)$.

The direct image $j_*\calT(\calX_U)$ of this sheaf on $B$ admits a Hodge structure of weight $3$ on its cohomology $H^1(B,j_*\calT(\calX_U)\otimes \QQ)$, defined as follows.  Indeed, suppose that $\VV$ is any $\QQ$-local system defined on a Zariski open set $j\colon U \hookrightarrow B$, which supports a variation of Hodge structure of weight $k$. Following Deligne \cite{tdh}, the associated Hodge filtration $\calF^i$ may be extended to coherent sheaves $\widetilde{\mathcal{F}}^i$ over the entire curve $B$ in a (quasi-)canonical manner. Zucker \cite{htdcl2cpm} then shows how to use the sheaves $\widetilde{\mathcal{F}}^i$ to equip the cohomology group $H^1(B,j_*\VV)$ with a pure polarizable Hodge structure of weight $k+1$. 

Our main use for this Hodge structure arises from the following result, which allows us to deduce information about the Hodge structure on $H^3(\calX,\CC)$ from that on $H^1(U,j_*\calT(\calX^U) \otimes \CC)$. In particular, it will allow us to compute the geometric genus of $\calX$, which we can use to restrict when $\calX$ may be Calabi-Yau.

\begin{lemma} \label{lem:hodgeinjection} Let $\pi\colon \calX \to B$ be a smooth threefold and suppose that the restriction $\pi^U\colon \calX^U \to U$ of $\pi$ to a Zariski open set $U \subset B$ is an $L$-polarized family of K3 surfaces, in the sense of Definition \ref{defn:L-pol}, where $L$ is the N\'{e}ron-Severi lattice of a general fibre of $\calX^U$. Then we have an injective morphism of Hodge structures
\[H^1(U,j_*\calT(\calX^U) \otimes \CC) \hookrightarrow H^3(\calX,\CC),\]
the cokernel of which is supported in the $(1,2)$ and $(2,1)$ components. In particular, $h^{3,0}$ of the Hodge structure on $H^1(U,j_*\calT(\calX^U))$ is equal to $h^{3,0}(\mathcal{X})$.
\end{lemma}
\begin{proof} Zucker \cite[Corollary 15.15]{htdcl2cpm} has proved that, in our setting, the Leray spectral sequence degenerates at the $E_2$ level. Since the generic fibre of $\pi$ is a K3 surface, we have $H^2(B,R^1\pi_*\CC) = 0$, so the Leray spectral sequence gives an exact sequence of morphisms of Hodge structures
\[0 \longrightarrow H^1(B,R^2\pi_*\CC) \longrightarrow H^3(\calX,\CC) \longrightarrow H^0(B,R^3\pi_*\CC) \longrightarrow 0.\]
Now, \cite[Proposition 15.12]{htdcl2cpm} gives a surjective map of sheaves
\[R^2\pi_*\CC \longrightarrow j_*R^2\pi^U_*\CC,\]
whose kernel is a skyscraper sheaf supported on $B \setminus U$. We thus have an isomorphism of cohomology $H^1(B,R^2\pi_*\CC) \cong H^1(B,j_*R^2\pi^U_*\CC)$ and, by Equation \eqref{eq:splitting}, the latter group is isomorphic to $H^1(U,j_*\calT(\calX_U) \otimes \CC)$, since our assumptions imply that $\calN\calS(\calX^U)$ is a trivial local system. This proves the first part of the lemma.

To prove the second part note that, by the exact sequence above, the cokernel of this morphism is $H^0(B,R^3\pi_*\CC)$. The Hodge structure on this group is analysed in detail by Zucker \cite[Section 15]{htdcl2cpm}. It follows from this analysis (for details, see the proof of \cite[Lemma 3.4]{cytfmqk3s}) that there is an injective morphism of Hodge structures of weight $3$
\[H^0(B,R^3\pi_*\CC) \hookrightarrow \bigoplus H^3(S_i,\CC),\]
where the direct sum runs over the irreducible components $S_i$ of the singular fibres in $\calX$. Since these components are all surfaces, the Hodge structure on $H^3(S_i,\CC)$ is supported in the $(1,2)$ and $(2,1)$ components. This completes the proof.
\end{proof}

Given this, we wish to find a way to study the local system $\calT(\calX^U)$. The aim of the next subsection is to prove a Hodge theoretic version of Theorem \ref{genfuninv} which will do just that: we will prove that the variation of Hodge structure on $\calT(\calX^U)$ is obtained as the pull-back of a given variation of Hodge structure $\mathbb{V}_n^+$ on $\calM_{M_n}$ by the generalized functional invariant. To construct this variation of Hodge structure, we first examine the structure of $\calM_{M_n}$.

\subsection{Modular curves and variations of Hodge structure}\label{sec:modularcurves}

For each $n \geq 2$, we may build a variation of Hodge structure on $\calM_{M_n}$ in the following way. Begin by choosing a basis $(e,f,g)$ for $M_n^{\perp} = H \oplus \langle 2n \rangle$, so that $f$ is a generator of $\langle 2n \rangle$ and $(e,g)$ is a basis for $H$ with $e^2 = g^2 = 0$ and $e.g = -1$; we make this slightly unusual choice for compatibility with the results of Dolgachev \cite{mslpk3s}. In the basis $(e,f,g)$, the pairing on $M_n^{\perp}$ therefore has matrix
\begin{equation}\label{Mnpairing} \left( \begin{matrix} 0 & 0 & -1 \\ 0 & 2n & 0 \\ -1 & 0 & 0 \end{matrix}\right).
\end{equation}

Define the period domain
\[\mathcal{D}_n : = \{ z \in \mathbb{P}(M_n^\perp \otimes \mathbb{C}) : \langle z ,z \rangle = 0, \langle z, \overline{z} \rangle > 0 \}.\]
There is an action on $\mathcal{D}_n$ by the group $\mathrm{O}(M_n^\perp)^*$, defined to be the kernel of the natural morphism $\mathrm{O}(M_n^\perp) \to \Aut(A_{M_n^\perp})$, where $A_{M_n^\perp}$ denotes the discriminant group. $\mathrm{O}(M_n^\perp)^*$ is generated by the group $\mathrm{SO}(M_n^\perp)^* := \mathrm{O}(M_n^\perp)^* \cap \mathrm{SO}(M_n^\perp)$ and the matrix
\[\iota = \left( \begin{matrix} 0 & 0 & -1 \\ 0 & 1 & 0 \\ -1 & 0 & 0 \end{matrix}\right),\]
acting on the basis $(e,f,g)$. Results of Dolgachev \cite[Section 3]{mslpk3s} show that there is an isomorphism between the the moduli space of $M_n$-polarized K3 surfaces and a dense open set in $\calD_n/\mathrm{O}(M_n^{\perp})^*$.

Now let $\mathcal{F}^2$ be the tautological vector bundle $\mathcal{O}_{\mathbb{P}^2}(-1)$ restricted to $\mathcal{D}_n$, which is a sub-bundle of the vector bundle $\calF^0 := (M_n^{\perp}\otimes \mathcal{O}_{\mathbb{P}^2})|_{\mathcal{D}_n}$ on $\mathcal{D}_n$. Setting $\calF^1$ to be the fibrewise orthogonal complement of $\calF^2$ in $\calF^0$, we obtain a weight $2$ variation of Hodge structure $\calF^2 \subseteq \calF^1 \subseteq \calF^0$, which descends to variations of Hodge structure $\VV_n$ and $\VV_n^+$ on the quotients $\calD_n/\mathrm{SO}(M_n^{\perp})^*$ and $\calD_n/\mathrm{O}(M_n^{\perp})^*$ respectively. These two variations of Hodge structure are closely related: there is a double covering $\calD_n/\mathrm{SO}(M_n^{\perp})^* \to \calD_n/\mathrm{O}(M_n^{\perp})^*$, and $\VV_n$ is the pull-back of $\VV_n^+$ under this covering map.

The following should be thought of as a Hodge theoretic version of Theorem \ref{genfuninv}.

\begin{proposition}\label{prop:pullback}
Suppose $n \geq 2$. Let $\pi^U\colon \calX^U \to U$ be a non-isotrivial $M_n$-polarized family of K3 surfaces over a quasi-projective curve $U$, such that the N\'{e}ron-Severi group of a general fibre of $\calX^U$ is isomorphic to $M_n$. Then the variation of Hodge structure on $\calT(\calX^U)$  is the pullback of $\VV_n^+$ by the generalized functional invariant map $g\colon U \to \calM_{M_n}$.
\end{proposition}
\begin{proof}  This proposition is proved in largely the same way as Theorem \ref{genfuninv} (and, by extension, \cite[Theorem 2.3]{cytfmqk3s}). Assume for a contradiction that the local systems $\calT(\calX^U)$ and $g^*(\VV_n^+)$ differ on $U$. Let $\{U_i\}$ denote an open cover of $U$, with each $U_i$ simply connected. As the local systems $\calT(\calX^U)$ and $g^*(\VV_n^+)$ are trivial on each $U_i$, they must differ in the way that they glue over the intersections $U_i \cap U_j$. Let $V$ be a connected component of such an intersection, so that the gluing maps differ over $V$. 

Over $V$, the two gluing maps must differ by a fibrewise automorphism in the group $\mathrm{O}(M_n^\perp)^*$. But it follows from a slight modification of the argument in \cite[Theorem 2.3]{cytfmqk3s} (see Theorem \ref{genfuninv}) that no such automorphism exists for $n > 1$. Thus $\calT(\calX^U)$ and $g^*(\VV_n^+)$ are isomorphic as local systems.

It remains to show that the variations of Hodge structure agree. To do this, it suffices to show that $g^*\calF^2(\VV^+) = \calF^2(\calT(\calX^U))$. But this is a standard property of the period map.
\end{proof}

In order to use this result to study $M_n$-polarized families of K3 surfaces, we need to understand the variations of Hodge structure $\VV_n^+$. As we shall see, it will turn out to be simpler to study the closely related variations of Hodge structure $\VV_n$, then use them to deduce the properties of $\VV_n^+$.

We begin by studying the quotients $\calD_n/\mathrm{SO}(M_n^{\perp})^*$ and $\calD_n/\mathrm{O}(M_n^{\perp})^*$. Dolgachev \cite{mslpk3s} has shown that  $\calD_n/\mathrm{O}(M_n^{\perp})^*$ may be realized as a certain modular curve. To define this curve we recall some standard notation. Let
\[\Gamma_0(n) :=  \left. \left\{ \gamma \in \text{SL}_2(\ZZ) : \gamma \equiv \left( \begin{matrix} * & * \\ 0 & * \end{matrix} \right) \bmod n \right\} \middle/ (\pm 1)  \subseteq \mathrm{PSL}_2(\ZZ) \right. .  \]
By convention, $\Gamma_0(1)$ is just the full modular group $\Gamma = \mathrm{PSL}_2(\ZZ)$. We also have  
\[\Gamma_0(n)^+:=  \Gamma_0(n) \cup \tau_n \Gamma_0(n) \subseteq \mathrm{PSL}_2(\RR)\] 
where 
\[\tau_n = \left( \begin{matrix} 0 & -1/\sqrt{n} \\ \sqrt{n} & 0 \end{matrix} \right)\]
is the \emph{Fricke involution}.

With this notation, define $Y_0(n) := \HH/\Gamma_0(n)$ and $Y_0(n)^+ := \HH/\Gamma_0(n)^+$, where $\HH$ denotes the upper half-plane in $\CC$, and let $i\colon Y_0(n) \hookrightarrow X_0(n) $ and $i^+\colon Y_0(n)^+ \hookrightarrow X_0(n)^+$ denote their compactifications. Dolgachev \cite[Theorem 7.1]{mslpk3s} has proved that there is an isomorphism $\calD_n/\mathrm{O}(M_n^{\perp})^* \cong Y_0(n)^+$ so, in particular, $X_0(n)^+$ defines a compactification $\calM_{M_n}$ of the moduli space of $M_n$-polarized K3 surfaces.

Following \cite[Section 7]{mslpk3s}, for $n \geq 2$ the quotient $\calD_n/\mathrm{SO}(M_n^{\perp})^*$ may also be presented as the modular curve $Y_0(n)$ in the following way. We begin with the well-known isomorphism
\[\mathrm{PSL}_2(\mathbb{R}) \longrightarrow \mathrm{PSO}(1,2).\]
In \cite[Remark 7.2]{mslpk3s}, Dolgachev presents this morphism as
\begin{align*}A\colon \mathrm{SL}_2(\RR) &\longrightarrow \mathrm{SO}(1,2)\\
\gamma = \left( \begin{matrix} a & b \\ c & d \end{matrix} \right) &\longmapsto A(\gamma):= \left( \begin{matrix}a^2 & -2nab & nb^2 \\ -ac/n & ad + bc & -bd \\ c^2/n & -2 cd & d^2\end{matrix}\right)\end{align*}
From this, we see that if $\gamma$ acts on $\RR^2$ with basis $(x,y)$, then $A(\gamma)$ is an expression for the action of $\gamma$ on $\mathrm{Sym}^2(\mathbb{R}^2)$ in terms of the basis $(x^2/n,-xy,y^2)$. Moreover, the matrix $A(\gamma)$ preserves the pairing \eqref{Mnpairing}, so $A(\gamma)$ defines an isometry of $M_n^{\perp} \otimes \RR$.

Let $\calD_n^+ \cong \HH$ denote a connected component of $\calD_n$ and let $\mathrm{SO}(M_n^{\perp})^{+}$ denote the index $2$ subgroup of $\mathrm{SO}(M_n^{\perp})^{*}$ that stabilizes $\calD_n^+$.
The map $A$ sends $\Gamma_0(n)$ to the group $\mathrm{SO}(M_n^\perp)^+$. A slight modification of the argument in \cite[Section 7]{mslpk3s} then shows that the quotient $\calD_n/\mathrm{SO}(M_n^{\perp})^* \cong \calD_n^+/\mathrm{SO}(M_n^{\perp})^{+}$ is isomorphic to the  modular curve $Y_0(n)$.

We may use this to explicitly construct the variation of Hodge structure $\VV_n$ on $\calD_n/\mathrm{SO}(M_n^{\perp})^* \cong Y_0(n)$.

\begin{proposition}\label{prop:symsquare}
The variation of Hodge structure $\mathbb{V}_n$ is the symmetric square of a polarizable weight one $\mathbb{Z}$-variation of Hodge structure $\VV_{\sqrt{n}}$ over $Y_0(n)$.
\end{proposition}

\begin{proof} Begin by considering the weight $1$ variation of Hodge structure $\calF^1 \subseteq \calF^0$ on $\mathbb{H} \times \mathbb{Z}^2$, defined so that $\calF^1$ is the sub-bundle of $\calF^0 := \mathbb{H} \times \CC^2$ whose fibre over $\tau \in \HH$ is given by the subspace in $ \mathbb{C}^2$ spanned by $(\tau,1)$.

Now define a local system  $\mathbb{V}_{\sqrt{n}}$ on $Y_0(n)$ as the quotient
\[\mathbb{H} \times \mathbb{Z}^2 / \Gamma_0(n)\] 
where $\Gamma_0(n)$ acts on $\mathbb{H}$ by M\"obius transformation and acts on $\mathbb{Z}^2$ by the natural representation. The variation of Hodge structure $\calF^1 \subseteq \calF^0$ descends under this quotient, so defines a weight $1$ variation of Hodge structure on $\mathbb{V}_{\sqrt{n}}$. 

We claim that $\mathrm{Sym}^2 \mathbb{V}_{\sqrt{n}}$ is actually the local system $\mathbb{V}_n$ described above. Indeed, this follows from the discussion above if we take the map
\begin{align*}\mathbb{H} &\longrightarrow \mathcal{D}_n^+ \subseteq \mathbb{P}(M_n^\perp \otimes \mathbb{C})\\
\tau &\longmapsto n\tau^2 e-\tau f+g
\end{align*}
which sends the vector $(\tau,1)$ to its symmetric square. This is a morphism of Hodge structures, and the map $A$ transforms the action of $\Gamma_0(n)$ on $\mathbb{H} \times \CC^2$  into the action of $\mathrm{SO}(M_n^\perp)^+$ on $(\calD_n^+,M_n^{\perp}\otimes \mathcal{O}_{\mathbb{P}^2})$.
\end{proof}

\begin{remark} \label{rem:iota} Note that a similar result does \emph{not} hold for the variation of Hodge structure $\mathbb{V}_n^+$. This is because the matrix $\iota \in \mathrm{O}(M_n^\perp)^*$ has no preimage under $A$ in $\mathrm{SL}_2(\mathbb{R})$, so the action of $\mathrm{O}(M_n^\perp)^*$ is \emph{not} the one induced by applying $A$ to $\Gamma_0(n)^+$. Indeed, in the basis $(e,f,g)$ the matrix of $\iota$ has determinant $-1$, so it cannot come from an element of $\mathrm{SL}_2(\mathbb{R})$.

In fact, one may check that the map $A$ sends
\[\overline{\tau}_n := \sqrt{-1} \left( \begin{matrix} 0 & -1/\sqrt{n} \\ \sqrt{n} & 0 \end{matrix} \right)\]
to $\iota$; this should be thought of as a complex rotation of the Fricke involution. Define
\[\overline{\Gamma_0(n)^+} := \Gamma_0(n) \cup \overline{\tau}_n \Gamma_0(n) \subseteq \mathrm{PSL}_2(\CC).\]
Then $\overline{\Gamma_0(n)^+}$ is isomorphic to $\Gamma_0(n)^+$, as a subgroup of $\mathrm{PSL}_2(\CC)$, and acts on $\HH$ in the same way. The image of $\overline{\Gamma_0(n)^+}$ under $A$ is the subgroup of $\mathrm{O}(M_n^{\perp})^*$ that stabilizes $\calD_n^+$. From this we recover the isomorphism $\calD_n/\mathrm{O}(M_n^{\perp})^* \cong Y_0(n)^+$.
\end{remark}

Now that we understand $\VV_n$, we may use it to deduce properties of $\VV_n^+$. In order to apply Lemma \ref{lem:hodgeinjection} and access Hodge theoretic data about threefolds fibred by $M_n$-polarized K3 surfaces, we need to understand the Hodge decompositions induced on the cohomology groups $H^1(X_0(n),i_*\VV_n)$ and $H^1(X_0(n)^+,i^+_*\VV^+_n)$. We begin with a technical lemma that will allow us to compare these cohomology groups.

\begin{lemma}\label{lemma:injectivity} Let $g\colon B \to X_0(n)^+$ be a covering of $X_0(n)^+$ by a smooth complete curve. Let $U \subset B$ denote the preimage of $Y_0(n)^+$ under $g$ and let $j \colon U \to B$ denote the embedding. Then the natural morphism of Hodge structures induced by the pull-back \textup{(}see \cite[Proposition 8.2]{htdcl2cpm}\textup{)}
\[g^*\colon H^1(X_0(n)^+, i^+_*\VV_n^+) \longrightarrow H^1(B,j_*g^*\VV_n^+)\]
is injective.
\end{lemma}
\begin{proof} Let $\widetilde{\calX}_n \to X_0(n)^+$ denote a modular family of $M_n$-polarized K3 surfaces with smooth total space over $X_0(n)^+ \cong \calM_{M_n}$; such a family may be obtained as a resolution of the family $\calX_n$ constructed in Section \ref{sec:modularinvariants}. Let $\calX'$ be a smooth resolution of the pull-back of $\widetilde{\calX}_n$ by the map $g$. We have a diagram
\[\xymatrix{H^3(\widetilde{\calX}_n,\CC) \ar[r]  & H^3(\calX',\CC) \\
H^1(X_0(n)^+, i^+_*\VV_n^+) \ar@{^{(}->}[u] \ar[r]^-{g^*} & H^1(B,j_*g^*\VV_n^+) \ar@{^{(}->}[u]
}\]
here the horizontal maps are the morphisms of Hodge structures induced by the pull-backs, and the vertical maps are the injective morphisms of Hodge structure arising from the Leray spectral sequence (see Lemma \ref{lem:hodgeinjection}).  It follows from functoriality of the Leray spectral sequence that the diagram commutes.

Using this diagram, we see that to prove injectivity of $g^*$, it suffices to prove that the map $H^3(\widetilde{\calX}_n,\CC) \to H^3(\calX',\CC)$ is injective. But this map is induced by the pull-back along the surjective map $\calX' \to \widetilde{\calX}_n$, obtained as the composition of the resolution and the pull-back by $g$, so its injectivity follows immediately from \cite[Lemma 7.28]{htcagi}. 
\end{proof}

Using this, we can prove an important result about the Hodge decompositions on the cohomology groups $H^1(X_0(n),i_*\VV_n)$ and $H^1(X_0(n)^+,i^+_*\VV^+_n)$

\begin{proposition}\label{prop:12=0}
The Hodge numbers $h^{p,q}$ of the weight $3$ Hodge structures on $H^1(X_0(n),i_*\VV_n)$ and $H^1(X_0(n)^+,i^+_*\VV_n^+)$ are trivial if $(p,q) = (2,1)$ or $(1,2)$.
\end{proposition}
\begin{proof}
Shimura \cite{sliaafa} has shown that there is a Hodge structure of weight $3$ on the space $S_2(\Gamma_0(n)) \oplus \overline{S_2(\Gamma_0(n))}$, where $S_2(\Gamma_0(n))$ denotes the space of $\CC^2$-valued cusp forms of weight $2$ for $\Gamma_0(n)$, which is concentrated in the $(0,3)$ and $(3,0)$ components. Zucker \cite[Section 12]{htdcl2cpm} proved that this Hodge structure agrees with the Hodge structure on $H^1(X_0(n),i_*\mathbb{V}_n)$ described above; this completes the proof in the $X_0(n)$ case. Given this, the $X_0(n)^+$ case follows immediately from Lemma \ref{lemma:injectivity}.
\end{proof}

We can use this to compute the Hodge numbers $h^{3,0}$ and $h^{0,3}$ of the Hodge structure on $H^1(X_0(n)^+,i^+_*\VV^+_n)$. Indeed, since we know that the Hodge numbers $h^{1,2}$ and $h^{2,1}$ are trivial, it suffices to compute the rank of $H^1(X_0(n)^+,i^+_*\mathbb{V}_n^+)$. To do this, we follow the method of \cite[Section 3.3]{cytfmqk3s}, which requires us to determine the local monodromy matrices of the local system $\mathbb{V}_n^+$. 

\subsection{Computation of the local monodromy}\label{sec:localmonodromy}

We begin by describing the local monodromy matrices of $\mathbb{V}_n$, which we compute using Proposition \ref{prop:symsquare} and the well-known local monodromy matrices of $\mathbb{V}_{\sqrt{n}}$.  For $n > 4$, we can use these to deduce the local monodromy matrices of $\mathbb{V}_n^+$. The cases $n \in \{2,3,4\}$ will then be tackled individually.

Let $p$ be a point in $\mathbb{H}$ that is fixed by a matrix $\gamma$ in $\mathrm{SL}_2(\mathbb{R})$. If the order of $\gamma$ is $4$, then $\gamma$ is conjugate to
\[\pm \left(\begin{matrix} 0 & 1 \\ -1 & 0 \end{matrix} \right),\]
whilst if the order of $\gamma$ is $3$ or $6$, then $\gamma$ is conjugate to 
\[\pm \left( \begin{matrix} 0 & 1 \\ -1 & -1 \end{matrix}\right)^i\]
for $i \in \{1,2\}$. Finally, for each cusp $p$ of $X_0(n)$, the matrix which fixes $p$ is conjugate to 
\[\pm \left(\begin{matrix} 1 & n \\ 0 & 1\end{matrix} \right)\]
for some integer $n$. Applying the map $A$ and reducing to Jordan form, we see that the local monodromy matrices of $\mathbb{V}_n$ are conjugate to
\[\left( \begin{matrix} {-1} & 0 & 0 \\ 0 & {-1} & 0 \\ 0 & 0 & 1 \end{matrix}\right),\ 
\left( \begin{matrix} \omega & 0 & 0 \\ 0 & \omega^2 & 0 \\ 0 & 0 & 1 \end{matrix}\right), \text{ or }
\left( \begin{matrix} 1 & 1 & 0 \\ 0 & 1 & 1 \\ 0 & 0 & 1 \end{matrix}\right),\]
where $\omega$ denotes a primitive cube root of $1$. 

Finally, in order to generate the entire group $\mathrm{O}(M_n^\perp)^*$, one must add the matrix $\iota$ (see Remark \ref{rem:iota}); it has Jordan normal form
\[\left( \begin{matrix} {-1} & 0 & 0 \\ 0 & {1} & 0 \\ 0 & 0 & 1 \end{matrix}\right).\]

Now we descend to the local monodromy matrices of $\mathbb{V}_n^+$. The following result seems well known.

\begin{lemma}\label{lemma:fixed}
The following statements hold:
\begin{enumerate}
\item If $n \neq 4$, the Fricke involution $\tau_n$ does not fix any cusp of $X_0(n)$.
\item If $n > 3$, the Fricke involution $\tau_n$ does not fix any elliptic orbifold point of $X_0(n)$.
\end{enumerate}
\end{lemma}
\begin{proof}
The first statement is precisely \cite[Proposition 3]{hmc} and the second follows easily from the discussion following it. In particular, \cite{hmc} shows that the fixed points of $\tau_n$ correspond to specific elliptic curves with level $n$ structure and with complex multiplication by $\mathbb{Q}(\sqrt{-n})$.
\end{proof}

One can check by concrete computation that
\begin{enumerate}
\item If $n=2$, the Fricke involution $\tau_2$ fixes one smooth point in $X_0(2)$ and the orbifold point of order $2$.
\item If $n=3$, the Fricke involution $\tau_3$ fixes one smooth point in $X_0(3)$ and the orbifold point of order $3$.
\item If $n=4$, the Fricke involution $\tau_4$ fixes one smooth point in $X_0(4)$ and a cusp of width $4$.
\end{enumerate}
Otherwise, Lemma \ref{lemma:fixed} shows that only smooth points of $X_0(n)$ are fixed by the Fricke involution $\tau_n$. 

It follows that if $n > 4$, the local monodromy matrices in the local system $\mathbb{V}_n^+$ are given by 
\[\left( \begin{matrix} {-1} & 0 & 0 \\ 0 & {1} & 0 \\ 0 & 0 & 1 \end{matrix}\right),\ \left( \begin{matrix} {-1} & 0 & 0 \\ 0 & {-1} & 0 \\ 0 & 0 & 1 \end{matrix}\right),\ 
\left( \begin{matrix} \omega & 0 & 0 \\ 0 & \omega^2 & 0 \\ 0 & 0 & 1 \end{matrix}\right),\
\left( \begin{matrix} 1 & 1 & 0 \\ 0 & 1 & 1 \\ 0 & 0 & 1 \end{matrix}\right)\]
at the images of a smooth fixed point, an orbifold point of order $2$, an orbifold point of order $3$, and a cusp respectively.

Moreover, the local monodromy matrices in the local system $\mathbb{V}_2^+$ were computed in \cite[Example 3.6]{cytfmqk3s}. They are given by
\[\left( \begin{matrix} {-1} & 0 & 0 \\ 0 & {1} & 0 \\ 0 & 0 & 1 \end{matrix}\right),\ 
\left( \begin{matrix} \sqrt{-1} & 0 & 0 \\ 0 & -\sqrt{-1} & 0 \\ 0 & 0 & -1 \end{matrix}\right),\
\left( \begin{matrix} 1 & 1 & 0 \\ 0 & 1 & 1 \\ 0 & 0 & 1 \end{matrix}\right)\]
around the orbifold point of order $2$, the orbifold point of order $4$, and the cusp respectively.

To compute the local monodromy matrices in the local systems $\mathbb{V}_3^+$ and $\mathbb{V}_4^+$, we use the fundamental modular families $\calX_3 \to \calM_{M_3}$ and $\calX_4 \to \calM_{M_4}$. These families were computed explicitly in \cite[Section 5.4]{flpk3sm} (see also Proposition \ref{prop:modularity}, below). Moreover, it follows from Proposition \ref{prop:pullback} that $\mathbb{V}_3^+$ and $\mathbb{V}_4^+$ agree with the transcendental variations of Hodge structure associated with these families.

The periods of the families $\calX_3$ and $\calX_4$ were computed by Doran and Malmendier \cite[Lemma 6.6]{cymrsrmt}, who found that they are given by the hypergeometric functions $\mbox{}_3F_2(\frac{1}{3},\frac{1}{2},\frac{2}{3};1,1,108\lambda)$ and $\mbox{}_3F_2(\frac{1}{2},\frac{1}{2},\frac{1}{2};1,1;64\lambda)$ respectively. From this, the global monodromy representations of the transcendental variations of Hodge structure $\calT(\calX_3)$ and $\calT(\calX_4)$ can be computed using a theorem of Levelt \cite[Theorem 1.1]{leveltthesis} (see also \cite[Theorem 3.5]{mhff}). This computation shows that the local monodromy matrices in the local system $\VV_3^+$ are conjugate to
\[\left( \begin{matrix} {-1} & 0 & 0 \\ 0 & {1} & 0 \\ 0 & 0 & 1 \end{matrix}\right),\ 
\left( \begin{matrix} \omega & 0 & 0 \\ 0 & \omega^2 & 0 \\ 0 & 0 & -1 \end{matrix}\right),\
\left( \begin{matrix} 1 & 1 & 0 \\ 0 & 1 & 1 \\ 0 & 0 & 1 \end{matrix}\right)\]
around the orbifold point of order $2$, the orbifold point of order $6$, and the cusp, respectively, and the local monodromy matrices in the local system $\VV_4^+$ are conjugate to
\[\left( \begin{matrix} {-1} & 0 & 0 \\ 0 & {1} & 0 \\ 0 & 0 & 1 \end{matrix}\right),\ 
\left( \begin{matrix} 1 & 1 & 0 \\ 0 & 1 & 1 \\ 0 & 0 & 1 \end{matrix}\right),\
\left( \begin{matrix} -1 & 1 & 0 \\ 0 & -1 & 1 \\ 0 & 0 & -1 \end{matrix}\right)\]
around the orbifold point of order $2$, the cusp of width $1$, and the cusp of width $2$, respectively.

Now we have the local monodromy data, we may easily compute the rank of $H^1(X_0(n)^+, i^+_*\mathbb{V}_n^+)$. In general, suppose that $\mathbb{V}$ is a local system on an open subset $j\colon U \hookrightarrow B$, and let $\{p_1,\dots, p_n\}= B \setminus U$. Let $\gamma_i$ be the monodromy transformation associated to a small loop moving counterclockwise around $p_i$ and let
\[R(p_i) = \rank (\mathbb{V}_{q_i}) - \rank(\mathbb{V}_{q_i}^{\gamma_i}),\]
where $\VV_{p_i}^{\gamma_i}$ is the subspace of elements of $\VV_{p_i}$ that are fixed under the action of $\gamma_i$. The rank of $H^1(B,j_* \mathbb{V})$ is given by the following variation of Poincar\'{e}'s formula in classical topology, due to del Angel, M\"{u}ller-Stach, van Straten and Zuo \cite[Proposition 3.6]{hca1pfcy3f}:
\begin{equation} \label{eq:poincare}h^1(B,j_*\mathbb{V}) = \sum_{i=1}^n R(p_i) + 2 (g(B) -1) \rank (\mathbb{V}).\end{equation}

Applying this in our setting, we see that $h^{1}(X_0(n)^+,i^+_*\mathbb{V}_n^+) = 0$ if $n \in \{2,3,4\}$ and 
\begin{equation} \label{eq:h1}
h^{1}(X_0(n)^+,i^+_*\mathbb{V}_n^+) = k_n+ \nu_2 +  \nu_3 + \nu_{\infty} + 6\left(g(X_0(n)^+) - 1\right), 
\end{equation} 
for $n>4$, where $k_n$ denotes the number of smooth points fixed by $\tau_n$ in $X_0(n)$, and $\nu_2$, $\nu_3$ and $\nu_{\infty}$ denote the numbers of elliptic points of order $2$, elliptic points of order $3$, and cusps in $X_0(n)$ respectively. Formulas for the various terms in this expression are well-known: formulas for $\nu_2$, $\nu_3$ and $\nu_{\infty}$ appear in \cite[Corollary 3.7.2 and Section 3.8]{fcmf}, and a formula for $k_n$ appears in \cite[Section 2]{hmc}.

\subsection{Application to Calabi-Yau threefolds}\label{sec:CYbound}

We wish to use this theory to restrict which lattices $M_n$ can arise as polarizations in a Calabi-Yau threefold fibred non-isotrivially by $M_n$-polarized K3 surfaces.

So let  $\pi\colon \calX \to B$ be a smooth Calabi-Yau threefold fibred non-isotrivially by $M_n$-polarized K3 surfaces (note that the cohomological condition in the definition of Calabi-Yau threefold necessitates $B \cong \PP^1$). Define $\pi^U\colon \calX^U \to U$ as before, and let $j\colon U \to B$ denote the inclusion.

%The fibration $\pi$ induces a weight two variation of Hodge structure $\calT(\calX^U)$ on the open set $U$. This variation of Hodge structure must satisfy the following condition:
%
%\begin{defn}
%A weight two variation of Hodge structure $\mathbb{V}$ over an open set $j\colon U \hookrightarrow B$ is of \emph{Calabi-Yau type} if the Hodge structure on $H^1(B,j_*\mathbb{V})$ has $h^{3,0} = 1$. 
%\end{defn}

By Proposition \ref{prop:pullback}, the variation of Hodge structure on $\calT(\calX^U)$ is the pullback of $\VV_n^+$ by the generalized functional invariant map $g\colon B \to X_0(n)^+$. Combining this with Lemmas \ref{lem:hodgeinjection} and \ref{lemma:injectivity}, we see that we must have an injective morphism of Hodge structures
\begin{equation}\label{eq:inj} H^1(X_0(n)^+,i_*^+\VV_n^+) \hookrightarrow H^3(\calX,\CC).\end{equation}
Moreover, by the Calabi-Yau condition, we must have $h^{3,0}(\calX) = 1$, so the Hodge structure on $H^1(X_0(n)^+,i_*^+\VV_n^+)$ must have Hodge number $h^{3,0} \leq 1$. By Proposition \ref{prop:12=0}, this is equivalent to the condition $\rank H^1(X_0(n)^+, i^+_*\mathbb{V}_n^+) \leq 2$.

%\begin{lemma} \label{lemma:gg}
%If $\calX$ is a smooth threefold fibred non-isotrivially by $M_n$-polarized K3 surfaces over a smooth curve $B$, for $n \geq 2$, then the geometric genus $h^{3,0}(\calX)$ is at least as large as $h^{3,0}_{M_n}$, where $h^{3,0}_{M_n}$ denotes the dimension of the $(3,0)$ part of the Hodge structure on $H^1(X_0(n)^+, i_* \mathbb{V}^{+}_n)$.
%\end{lemma}
%
%\begin{proof} As usual, we let $\calX_n$ denote the modular family of $M_n$-polarized K3 surfaces over $\calM_{M_n} \cong X_0(n)^+$. Lemma \ref{lem:hodgeinjection} implies that $h^{3,0}(\calX_n) = h^{3,0}_{M_n}$ for all $n \geq 2$.
%
%Now, by Theorem \ref{genfuninv}, $\calX$ is birational to the pull-back ${\calX}_{n,g}$ of one of the families $\calX_n$ along the generalized functional invariant map $g\colon B \rightarrow {\mathcal{M}}_{M_n}$. Let $\widetilde{\calX}_{n,g}$ denote a smooth resolution of ${\calX}_{n,g}$.  Then there is a surjective map $\widetilde{\calX}_{n,g} \rightarrow \widetilde{\calX}_n$, obtained as the composition of the resolution and the pull-back. It follows from \cite[Lemma 7.28]{htcagi} that the pull-back along this morphism is injective on the level of Hodge structures and therefore, in particular, that the geometric genus $h^{3,0}(\widetilde{\calX}_{n,g})$ is at least as large as that of ${\mathcal{X}}_n$. Since the geometric genus is a birational invariant, we must also have that $h^{3,0}(\calX)$ is at least as large as $h^{3,0}({\mathcal{X}}_n)$. This completes our proof.
%\end{proof}

Summarizing, we see that in order for $\calX$ to be a smooth Calabi-Yau, the following two conditions must hold:
\begin{enumerate}
\item The curve $X_0(n)^+$ must have genus $0$ (by Hurwitz's theorem), and
\item The local system $\mathbb{V}_n^+$ must satisfy $\rank H^1(X_0(n)^+, i^+_*\mathbb{V}_n^+) \leq 2$.
\end{enumerate}

The set of $X_0(n)^+$ that have genus $0$ may be computed using the double cover $X_0(n) \to X_0(n)^+$ induced by the Fricke involution $\tau_n$ and work of Ogg \cite{hmc} on the involutions of $X_0(n)$. We note first that $X_0(n)^+$ must have genus $0$ if $X_0(n)$ has genus $0$, and a complete list of all integers $n$ so that $X_0(n)$ has genus $0$ is known:
\[\{1, 2,3,4,5,6,7,8,9,10, 12,13, 16, 18, 25\}.\]
Secondly, $X_0(n)$ has genus $1$ if and only if $n$ is an element of the set
\[\{11, 14, 15, 17, 19, 20, 21, 24, 27, 32, 36, 49\},\]
and in all of these cases it is easy to check, using the Hurwitz formula and the formula for the number of fixed points $k_n$ of the Fricke involution from \cite[Section 2]{hmc}, that $X_0(n)^+$ has genus $0$. Finally, if $X_0(n)$ has genus greater than $1$, then $X_0(n)^+$ has genus $0$ if and only if the Fricke involution is a hyperelliptic involution, and the set of all $n$ such that this holds was computed in \cite[Theorem 2]{hmc} to be
\[\{23, 26, 29, 31, 35, 39, 41, 47, 50, 59, 71\}.\]
We thus find that $X_0(n)^+$ has genus $0$ if and only if $n$ is an element of one of these three sets; i.e. when $n$ lies in the set
\begin{align*}\{1,2,3,4,5,6,7,8,9,10, 11,12,13,14, 15,16,  17,18, 19, 20, 21, \\
23, 24, 25,  26, 27, 29, 31, 32, 35, 36, 39, 41, 47, 49, 50, 59, 71 \}.\end{align*}

\begin{theorem}\label{thm:H1CY3}
Suppose $n \geq 2$ and $X_0(n)^+$ is rational. $H^1(X_0(n)^+,i^+_*\mathbb{V}_n^+) =0$ if and only if $n$ is an element of the set
\[\{2,3,4,5,6,7,8,9,11\}\]
and $H^1(X_0(n)^+,i^+_*\mathbb{V}_n^+) =2$ if and only if $n$ is an element of the set
\[\{10,12,13,14,15,16,17,19,23\}.\]
In particular, if $n \neq 1$ and $\calX$ is a smooth Calabi-Yau threefold fibred non-isotrivially by $M_n$-polarized K3 surfaces, $n$ must be one of the integers contained in the sets above.
\end{theorem}
\begin{proof} This is easy to check explicitly using Equation \eqref{eq:h1}. \end{proof}

Recall that a smooth Calabi-Yau threefold $\calX$ is called \emph{rigid} if $h^{2,1}(\calX) = 0$.

\begin{theorem}\label{thm:rigid}
If $\calX$ is a smooth Calabi-Yau threefold fibred non-isotrivially by $M_n$-polarized K3 surfaces and $n$ is an element of the set \[\{10,12,13,14,15,16,17,19,23\},\] 
then $\calX$ is a rigid Calabi-Yau threefold.
\end{theorem}
\begin{proof}  We begin by showing that any small deformation of $\calX$ must also admit a fibration by $M_n$-polarized K3 surfaces. The first step in this is to show that any small deformation of $\calX$ must contain a K3 surface. This will follow if we can show that any small deformation of $\calX$ lifts to a small deformation of the pair $(\calX,S)$, where $S$ is a general fibre of the K3 fibration on $\calX$.

 Small deformations of $\calX$ are parametrized by the cohomology group $H^1(\calX,T_{\calX})$, where $T_{\calX}$ denotes the tangent bundle to $\calX$, and small deformations of the pair $(\calX,S)$ are parametrized by $H^1(\calX,T_{\calX}(-\log S))$. We may relate these two groups as follows. The short exact sequence
\[0 \longrightarrow T_{\calX}(-\log S) \longrightarrow T_{\calX} \longrightarrow N_{S|\calX} \longrightarrow 0,\]
where $N_{S|\calX}$ is the normal bundle of $S$ in $\calX$, induces a long exact sequence in cohomology
\[\cdots \longrightarrow H^1(\calX,T_{\calX}(-\log S)) \stackrel{\alpha}{\longrightarrow} H^1(\calX,T_{\calX}) \stackrel{\beta}{\longrightarrow} H^1(S,N_{S|\calX}) \longrightarrow \cdots.\]
The map $\alpha$ in this exact sequence corresponds to the fact that any small deformation of $(\calX,S)$ induces a small deformation of $\calX$. Thus the obstruction to lifting a small deformation of $\calX$ to a small deformation of $(\calX,S)$ comes from the map $\beta$. But, since $S$ is a fibre in a K3 fibration on $\calX$, we have $N_{S|\calX} \cong \calO_S$, so $H^1(S,N_{S|\calX}) = 0$. Therefore $\beta$ is trivial, and any small deformation of $\calX$ lifts to a small deformation of the pair $(\calX,S)$.

Next we show that if $S$ is a K3 surface in a Calabi-Yau threefold $\calX$, then $S$ is the class of a fibre of a map $\pi\colon \calX \rightarrow \mathbb{P}^1$. Indeed, let $S$ be a smooth K3 surface in $\calX$ and let $i\colon S \hookrightarrow \calX$ be the embedding. Any section $s$ in $H^0(\calX,\calO_{\calX}(S))$ gives rise to a short exact sequence
\[0 \longrightarrow \calO_{\calX} \stackrel{s}{\longrightarrow}\calO_{\calX}(S) \longrightarrow i_*\calO_S(S) \longrightarrow 0\]
from which we obtain the exact sequence of cohomology groups
\[0 \longrightarrow H^0(\calX,\calO_{\calX}) \stackrel{s}{\longrightarrow} H^0({\calX},\calO_{\calX}(S)) \longrightarrow H^0(S,\calO_S(S)) \longrightarrow 0.\]
Now, by adjunction we have 
\[\calO_S(S) = \omega_{\calX}(S)|_S = \omega_S = \calO_S,\]
thus for a generic global section $s$ of $\calO_{\calX}(S)$, the restriction of $s$ to $S$ does not vanish. This implies that the base locus of the linear system $|S|$ in $\calX$ is empty. Moreover, the exact sequence above gives $h^0(\calX,\calO_{\calX}(S)) = 2$, so the morphism defined by the linear system $|S|$ maps to $\PP^1$ and has fibres deformation equivalent to $S$. It thus defines a K3 fibration on $\calX$.

Finally, we show that this K3 fibration is $M_n$-polarized. Let $G_S$ denote the subgroup of $H^2(S,\mathbb{Z})$ obtained as the pull-back of $H^2(\calX,\ZZ)$ under the inclusion of $S$ into $\calX$, equipped with the usual cup-product pairing. $G_S$ is a topological invariant of the pair $(\calX,S)$. Moreover, $G_S$ is precisely the lattice polarization of $\pi\colon \calX\rightarrow \mathbb{P}^1$. This shows that any small deformation of $\calX$ remains fibred by K3 surfaces and, more importantly, this K3 surface fibration is $M_n$-polarized. 

%Now, by the proof of Lemma \ref{lem:gg}, we know that $\calX$ must be birational to the pull-back ${\calX}_{n,g}$ of one of the modular families $\calX_n$, and that $H^1(X_0^+(n),i^+_*\mathbb{V}_n^+)$ is contained in $H^3(\widetilde{\calX}_{n,g},\mathbb{Q})$. Moreover, by Theorem \ref{thm:H1CY3} we have $H^1(X_0^+(n),i^+_*\mathbb{V}_n^+) =2$ and, by Proposition \ref{prop:12=0}, we see that the Hodge structure on $H^1(X_0^+(n),i^+_*\mathbb{V}_n^+)$ is concentrated in the $(3,0)$ and $(0,3)$ components. 
%
%By \cite{tvwft}, the birational map between $\calX$ and $\widetilde{\calX}_{n,g}$ may be decomposed as a composition of blow-ups and blow-downs between smooth varieties. By \cite[Theorem 7.31]{htcagi}, these maps induce morphisms of Hodge structures that are isomorphisms on the $(3,0)$ and $(0,3)$ components. It follows that $H^1(X_0^+(n),i^+_*\mathbb{V}_n^+)$ is also contained in $H^3({\calX},\mathbb{Q})$.

Now, by Equation \eqref{eq:inj}, we have an injective morphism of Hodge structures $H^1(X_0(n)^+,i^+_* \mathbb{V}^+_0) \hookrightarrow H^3(\calX,\mathbb{C})$. Moreover, it follows from Theorem \ref{thm:H1CY3} that $H^1(X_0(n)^+,i^+_*\mathbb{V}_n^+) =2$ and, by Proposition \ref{prop:12=0}, we see that the Hodge structure on $H^1(X_0(n)^+,i^+_*\mathbb{V}_n^+)$ is concentrated in the $(3,0)$ and $(0,3)$ components. Thus, since $\calX$ is a Calabi-Yau threefold, it follows that there is some sub-Hodge structure $\mathcal{H}$ of $H^3(\calX,\mathbb{C})$ with components only of types $(1,2)$ and $(2,1)$, so that  $H^3(\calX, \mathbb{C})$ is isomorphic to $H^1(X_0(n)^+,i^+_* \mathbb{V}^+_0) \oplus \mathcal{H}$.

Such a decomposition of Hodge structures also exists for a generic deformation of $\calX$, as such deformations are also fibred by $M_n$-polarized K3 surfaces. However, Bardelli has proven \cite[Proposition 1.1.2]{gghcfttcb} that this implies that a generic deformation has trivial Kodaira-Spencer map; in other words, $\calX$ must be rigid by the local Torelli theorem for Calabi-Yau manifolds.
\end{proof}

\begin{corollary} \label{cor:nonrigid}
If $\calX$ is a smooth, non-rigid Calabi-Yau threefold fibred non-isotrivially by $M_n$-polarized K3 surfaces, $n$ must be an element of the set 
\[\{1 ,2, 3, 4, 5, 6, 7, 8, 9, 11\}.\]
\end{corollary}

\begin{remark} Note that several examples of smooth, non-rigid Calabi-Yau threefolds fibred non-isotrivially by $M_1$-polarized K3 surfaces are given in \cite[Theorem 5.20]{flpk3sm}. The remaining cases will be addressed in the rest of this paper.
\end{remark}

%\begin{remark} Let $\calX$  be a smooth Calabi-Yau threefold fibred non-isotrivially by $M_n$-polarized K3 surfaces, for $n$ in the list from Theorem \ref{thm:rigid}. The Hodge structure on $H^3(\calX,\ZZ)$ can be used to build a weight $1$ Hodge structure with rank $2$ in a natural way. Of course, this is precisely the data of a complex elliptic curve. Furthermore, as in the proof of Proposition \ref{prop:12=0}, Zucker \cite[Section 12]{htdcl2cpm} tells us that the Hodge structure that it spans is just the Hodge structure of Shimura  \cite{sliaafa}. Thus the periods of these rigid Calabi-Yau varieties may be identified with periods of cusp forms for the group $\Gamma_0(n)$ which behave appropriately with respect to the action of the matrix
%\[\left(\begin{matrix}0 & -1/\sqrt{-n} \\ \sqrt{-n} & 0 \end{matrix} \right).\]
%\end{remark}

\begin{remark} \label{rem:mirrors}
Recall that one may associate two invariants to a smooth Fano threefold $Y$: the \emph{degree} $\delta := (-K_Y)^3$ and the \emph{index} $\rho$, defined to be the largest integer such that there exists an ample divisor $H$ with $-K_Y \sim \rho H$ in $\Pic(Y)$. A third invariant, which we call the \emph{hypersurface degree}, may be formed as the combination
\[n := \frac{\delta}{2\rho^2}.\]
$2n$ may be interpreted as the degree of a generic anticanonical K3 hypersurface in $Y$, with the polarization induced by $H$.

Using the well-known classification by Iskovskih \cite{f3f1, f3f2}, one may easily compute the hypersurface degrees for all smooth Fano threefolds of Picard rank $1$ (c.f. \cite[Theorem 2.1]{cpmd}). The set of integers one obtains are precisely those appearing in Corollary \ref{cor:nonrigid}. We postulate that this is not a coincidence. Indeed, if $Y$ is a Fano threefold of Picard rank $1$ and hypersurface degree $n$, then a generic anticanonical hypersurface in $Y$ is a smooth K3 surface with Picard group of rank $1$ and Picard lattice isomorphic to $\langle 2n \rangle$. According to Dolgachev \cite{mslpk3s}, the mirror of a K3 surface with Picard group isomorphic to $\langle 2n \rangle$ is an $M_n$-polarized K3 surface. In other words, the $M_n$-polarized K3 surfaces which appear as the fibres of  smooth, non-rigid Calabi-Yau threefolds fibred non-isotrivially by $M_n$-polarized K3 surfaces are precisely those which are mirror to anticanonical K3 surfaces in smooth Fano threefolds of Picard rank $1$. This mirror correspondence is closely related to the Fano-LG correspondence for smooth Fano threefolds of Picard rank $1$, it is explored further in \cite{mstdfcym,mpcytmpqft} (see also Remark \ref{rem:mirrors2}, below).
\end{remark}

\subsection{A further restriction}\label{sec:furtherrestriction}

In order to achieve our goal of completely classifying non-rigid Calabi-Yau threefolds admitting non-isotrivial fibrations by $M_n$-polarized K3 surfaces with $n \geq 2$, we need to place some additional restrictions on the form of the generalized functional invariant maps $g \colon \PP^1 \to \calM_{M_n}$ that may arise.

We begin by setting up some extra notation to describe these generalized functional invariants. Note first that, for $n \in \{2,3,4,5,6,7,8,9,11\}$, the moduli space $\calM_{M_n}$ has orbifold points of types $(2,\ldots,2,a,\infty)$, where $\infty$ denotes a cusp (which has order $\infty$) and $a \geq 2$ is an integer or $\infty$ (in the case where $\calM_{M_n}$ contains two cusps). A complete list of these orbifold points is given later, in Table \ref{tab:orbipoints}. Fix a parameter $\lambda$ on $\calM_{M_n}$ so that the $a$-orbifold point occurs at $\lambda = 0$, the cusp occurs at $\lambda = \infty$, and the $2$-orbifold points occur at $\lambda = \lambda_1,\ldots,\lambda_q$. Note that in the case $n=4$, where we have two cusps of different widths, the cusp of width $1$ occurs at $\lambda = \infty$ and the cusp of width $2$ occurs at $\lambda = 0$.

Now suppose that $g \colon \PP^1 \to \calM_{M_n}$ is a $d$-fold cover and let $[x_1,\ldots,x_k]$, $[y_1,\ldots,y_l]$, $[z_{1,1},\ldots,z_{m_1,1}],\ldots,[z_{1,1},\ldots,z_{m_q,q}]$ denote the partitions of $d$ that encode the ramification profiles of $g$ over $\lambda = \infty$ (the cusp), $\lambda = 0$ (the $a$-orbifold point) and $\lambda = \lambda_1,\ldots,\lambda_q$ (the $2$-orbifold points) respectively. Let $r$ denote the degree of ramification of $g$ away from $\lambda \in \{0,\infty,\lambda_1,\ldots,\lambda_q\}$, defined as
\[r := \sum_{\mathclap{\substack{p \in \PP^1 \\ g(p) \notin  \{0,\infty,\lambda_1,\ldots,\lambda_q\}}}} (e_p -1),\]
where $e_p$ denotes the ramification index of $g$ at the point $p \in \PP^1$.

The Hurwitz formula immediately implies that the relation
\begin{equation}\label{eq:hurwitz} k+l+m_1+\cdots + m_q-qd-r-2=0\end{equation}
holds between these variables. Moreover, we have:

\begin{theorem} \label{thm:partitions} Suppose that $\calX$ is a smooth non-rigid Calabi-Yau threefold fibred non-isotrivially by $M_n$-polarized K3 surfaces, with $n \geq 2$. Then the ramification profile $[y_1,\ldots,y_l]$ of the generalized functional invariant map of $\calX$ over the $a$-orbifold point in $\calM_{M_n}$ must be one of the possibilities listed in Table \ref{tab:partitions}.
\end{theorem}

\begin{table}
\begin{tabular}{|c|l|}
\hline
$n$ & Partitions \\
\hline
$2$ & $[1,1]$, $[1,2]$, $[1,3]$, $[1,4]$, $[2,2]$, $[2,3]$ $[2,4]$, $[3,3]$, $[3,4]$ $[4,4]$,$[5]$, $[6]$, $[7]$, $[8]$ \\
$3$ & $[1,1]$, $[1,2]$, $[1,3]$, $[2,2]$, $[2,3]$, $[3,3]$, $[4]$, $[5]$, $[6]$ \\
$4$ & $[1,1]$, $[1,2]$, $[2,2]$, $[3]$, $[4]$ \\
$5$ & $[1,1]$, $[1,2]$, $[2,2]$, $[3]$, $[4]$ \\
$6$ & $[1,1]$, $[2]$ \\
$7$ & $[1,1]$, $[2]$, $[3]$ \\
$8$ & $[1,1]$, $[2]$ \\
$9$ & $[1,1]$, $[2]$ \\
$11$ & $[1,1]$, $[2]$ \\
\hline
\end{tabular}
\caption{Allowable partitions $[y_1,\ldots,y_l]$ of $d$}
\label{tab:partitions}
\end{table}

\begin{proof} This is a simple exercise using \cite[Corollary 4.4]{hnpfe} (see also \cite[Section 4.1]{hnpfe}).
\end{proof}

\section{Modular families of lattice polarized K3 surfaces}\label{sec:modularfamilies}

Next we turn our attention to the converse question: when do smooth non-rigid Calabi-Yau threefolds fibred non-isotrivially by $M_n$-polarized K3 surfaces satisfying the conditions of the previous section actually exist? To answer this question, we will attempt to find an explicit construction for these threefolds.

The construction of such Calabi-Yau threefolds is a subject to which we have already given some attention: several explicit examples for $2 \leq n \leq 4$ are studied in \cite[Theorem 5.10]{flpk3sm} and the case $n = 2$ is treated in detail in \cite{cytfmqk3s}.

By Theorem \ref{genfuninv}, if $\pi\colon \calX \to B$ is such a Calabi-Yau threefold, the family $\pi^U\colon \calX^U \to U$ is given by the pull-back of the modular family $\calX_n \to \calM_{M_n}$ (constructed in Section \ref{sec:modularinvariants}) by the generalized functional invariant map $g \colon U \to \calM_{M_n}$. Thus in order to construct $\calX$, we begin by studying the families $\calX_n$.

\subsection{The modular families $\calX_n$ for small $n$} \label{sec:fundfam}

We will show that the modular families $\calX_n$ agree with a set of families first computed by Przyjalkowski \cite{wlgmsft} in the context of Landau-Ginzburg models; these are defined as the maximal crepant resolutions of the families $\bar{\calX}_n$ given explicitly in Table \ref{tab:families}. In this table, the ``Ambient space'' column gives the ambient space of the fibres and $\lambda \in \CC \cup \{\infty\}$ denotes a parameter on the base $\calM_{M_n}$ of the family, chosen so that the $a$-orbifold point and the cusp occur at $\lambda = 0$ and $\lambda = \infty$ respectively. For future reference, we also record the singularities occurring in the generic fibres of the families $\bar{\calX}_n$ in Table \ref{tab:Mnsingularities}.

\begin{table}
\begin{tabular}{|c|c|l|}
\hline
$n$ & Ambient space & Family \\
\hline
$2$ &$\PP^3[x,y,z,t]$& $t^4 + \lambda xyz(x+y+z-t) = 0$ \\
$3$ & $\PP^1[r,s] \times \PP^2[x,y,z]$& $s^2z^3 + \lambda r(r-s)xy(z-x-y) = 0$ \\
$4$ & $\prod_{i=1}^3\PP^1[r_i,s_i]$ & $s_1^2s_2^2s_3^2 - \lambda r_1(s_1-r_1)r_2(s_2-r_2)r_3(s_3-r_3) = 0$ \\
$5$ & $\PP^3[x,y,z,t]$ & $(x^2+xy + xz + xt + yz + yt + zt)^2 - \lambda xyzt = 0$ \\
$6$ & $\PP^3[x,y,z,t]$ & $(x+z+t)(x+y+z+t)(z+t)(y+z) - \lambda xyzt = 0$ \\
$7$ & $\PP^3[x,y,z,t]$ & $ (x+y+z+t)(y+z+t)(z+t)^2+$ \\ 
& & \hfill$+ (x+y+z+t)^2yz -  \lambda xyzt = 0$\\
$8$ & $\PP^3[x,y,z,t]$ & $(x+y+z+t)(x+t)(y+t)(z+t) - \lambda xyzt = 0$ \\
$9$ & $\PP^3[x,y,z,t]$ & $(x+y+z)(xt^2+yt^2+zt^2+xyt+xzt+yzt+xyz) -$ \\
& &\hfill $-\lambda xyzt = 0$ \\
$11$ & $\PP^3[x,y,z,t]$ &$(z+t)(x+y+t)(xy+zt)+x^2y^2+xyz^2+3xyzt -$ \\
& & \hfill$- \lambda xyzt = 0$ \\
\hline
\end{tabular}
\caption{The families $\bar{\calX}_n$}
\label{tab:families}
\end{table}

\begin{table}
\begin{tabular}{|c|l|}
\hline
$n$ & Singularities \\
\hline
$2$ & $6 \times A_3$ \\
$3$ &  $3 \times A_1$, $6 \times A_2$\\
$4$ & $12 \times A_1$\\
$5$ &  $3 \times A_1$, $3 \times D_4$ \\
$6$ &  $3 \times A_1$, $2 \times A_2$, $2 \times A_3$\\
$7$ &  $1\times A_1$, $3 \times A_2$, $1 \times A_3$, $1 \times A_4$\\ 
$8$ &  $6 \times A_1$, $3 \times A_2$\\
$9$ &  $3 \times A_1$, $3 \times A_2$, $1 \times D_4$ \\
$11$ & $2\times A_1$, $1 \times A_2$, $2 \times A_3$ \\
\hline
\end{tabular}
\caption{Singularities of a generic fibre of $\bar{\calX}_n$.}
\label{tab:Mnsingularities}
\end{table}

Let $\hat{\calX}_n \to \calM_{M_n}$ denote the maximal crepant resolution of $\bar{\calX}_n$. The singular fibres of $\hat{\calX}_n$  occur over the orbifold points of $\calM_{M_n}$. The types of these orbifold points are shown in the ``Orbifold type'' column of Table \ref{tab:orbipoints}.  The locations of the $2$-orbifold points $\lambda_1,\ldots,\lambda_q$ in these families are listed in the $\lambda_i$ column of Table \ref{tab:orbipoints}.

\begin{table}
\begin{tabular}{|c|c|l|}
\hline
$n$ & Orbifold type & $\lambda_i$ \\
\hline
$2$ &$(2,4,\infty)$& $256$ \\
$3$ & $(2,6,\infty)$& $108$ \\
$4$ & $(2,\infty,\infty)$ & $64$ \\
$5$ & $(2,2,2,\infty)$ & $22 + 10\sqrt{5}$, $22 - 10\sqrt{5}$ \\
$6$ & $(2,2,\infty,\infty)$ & $17 + 12\sqrt{2}$, $17 - 12\sqrt{2}$ \\
$7$ & $(2,2,3,\infty)$ & $-1$, $27$ \\ 
$8$ & $(2,2,\infty,\infty)$ & $12 + 8\sqrt{2}$, $12 - 8\sqrt{2}$ \\
$9$ & $(2,2,\infty,\infty)$ & $9 + 6\sqrt{3}$, $9 - 6\sqrt{3}$ \\
$11$ & $(2,2,2,2,\infty)$ &Roots of the cubic $\lambda^3-20\lambda^2+56\lambda-44$ \\
\hline
\end{tabular}
\caption{Orbifold points of $\calM_{M_n}$.}
\label{tab:orbipoints}
\end{table}

Let $U_{M_n}$ denote the open subset of $\calM_{M_n}$ obtained by removing these orbifold points. The restriction $\hat{\calX}_{n}^U$ of $\hat{\calX}_n$ to $U_{M_n}$ is a family of smooth K3 surfaces. The following proposition shows that these families are the fundamental modular families $\calX_n \to \calM_{M_n}$ that we have been seeking. In light of this result, in the following sections we will omit the hat and simply denote these families by $\calX_n$, and their restrictions to $U_{M_n}$ by $\calX_n^U$.

\begin{proposition} \label{prop:modularity} For $n \in \{2,3,4,5,6,7,8,9,11\}$, the families $\hat{\calX}_n^U \to U_{M_n}$ defined as above are $M_n$-polarized families of K3 surfaces in the sense of Definition \ref{defn:L-pol}. Moreover, these families are modular, in the sense that their generalized functional invariant maps are isomorphisms. 
\end{proposition}

The rest of this subsection is devoted to a proof of this proposition. For $n \in \{2,3,4\}$, Proposition \ref{prop:modularity} was proved in \cite[Section 5.4]{flpk3sm}. So suppose that $n \in \{5,6,7,8,9,11\}$. In these cases, Doran et al. \cite{mft} have shown that  $\hat{\calX}_n^U \to U_{M_n}$ is an $M_n$-polarized family of K3 surfaces. It therefore only remains to prove modularity of these families.

To do this, we begin by showing that, if $p \in \calM_{M_n}$ is an orbifold point with order $k$ (as given in Table \ref{tab:orbipoints}), the local monodromy of the transcendental local system $\calT(\hat{\calX}_n^U)$ has order $k$ around $p$.

To compute the orders of the local monodromies of $\calT(\hat{\calX}_n^U)$, we turn to the properties of the families $\hat{\calX}_n^U$. Przyjalkowski \cite{wlgmsft} constructs these families as \emph{weak Landau-Ginzburg models} of certain rank $1$ Fano threefolds. In particular, he shows that the holomorphic periods of the families $\hat{\calX}_n^U$ are equal to the $I$-series of the corresponding Fano threefolds. It follows that the regularized quantum differential operators associated to these Fano threefolds, which annihilate the $I$-series, are Picard-Fuchs differential equations for the families $\hat{\calX}_n^U$. These regularized quantum differential operators are computed for all rank $1$ Fano threefolds by Golyshev \cite{cpmd}; the family $\hat{\calX}_n^U$ corresponds to row $n$ of the $d=1$ table in \cite[Section 5.8]{cpmd}.

Once one has the Picard-Fuchs differential equations for the families $\hat{\calX}_n^U$, computation of the local monodromies of $\calT(\hat{\calX}_n^U)$ is a straightforward application of the Frobenius method. Combining the result of this calculation with the computation of the local monodromies of the local systems $\VV^+_n$ performed in Section \ref{sec:localmonodromy}, we conclude that $\calT(\hat{\calX}_n^U)$ and $\VV_n^+$ have the same \emph{monodromy profiles} for all $n \in \{5,6,7,8,9,11\}$.

\begin{defn} Let $\VV$ be a local system. The \emph{monodromy profile} of $\VV$ is the unordered list of the orders of the local monodromy matrices of $\VV$ (excluding, of course, points where this monodromy has order $1$).
\end{defn}

Moreover, we know from Proposition \ref{prop:pullback} that $\calT(\hat{\calX}_n^U)$ is the pull-back  of $\VV_n^+$ under the generalized functional invariant map $g$ of $\hat{\calX}_n^U$. Proposition \ref{prop:modularity} will therefore follow if we can prove that $g^*\VV^+_n$ can only have the same monodromy profile as $\VV_n^+$ (for $n \in \{5,6,7,8,9,11\}$) if $g$ is an isomorphism.

We will establish this fact through a pair of lemmas. The first lemma deals with the cases $n \in \{6,8,9\}$.

\begin{lemma}\label{lemma:2infinite}
Let $\mathbb{V}$ be a local system on a Zariski open subset of $\mathbb{P}^1$, which has at least two local monodromy matrices of infinite order and at least one other nontrivial local monodromy matrix. Let $g \colon \mathbb{P}^1 \rightarrow \mathbb{P}^1$ be a morphism. Then ${g}^* \mathbb{V}$ has the same monodromy profile as $\mathbb{V}$ if and only if $g$ is an isomorphism.
\end{lemma}
\begin{proof}
By the Riemann-Hurwitz formula, our hypotheses imply that
\[-2 = -2d + \sum_{p \in \PP^1}(e_p- 1),\]
where $d = \deg (g)$ and $e_p$ is the ramification index of $g$ at $p$. Let $\Sigma^\infty_g$ be the set of points $p \in \mathbb{P}^1$ such that the local monodromy of $\mathbb{V}$ around $g(p)$ is of infinite order. Then ${g}^* \mathbb{V}$ has $|\Sigma^\infty_g|$ points whose local monodromy matrices are of infinite order. If ${g}^* \mathbb{V}$ has the same number of points with monodromy of infinite order as $\mathbb{V}$, it follows that $g$ must be totally ramified at each point in $\Sigma_g^\infty$. Therefore, if there are $N$ points of $\mathbb{P}^1$ around which the local monodromy of $\mathbb{V}$ is infinite order, then
\[2(d-1) = N(d-1) + \sum_{p \in \PP^1 \setminus \Sigma_g^\infty} (e_p - 1).\]
Thus if $N>2$, we must have $d = 1$, so $g$ is an isomorphism. Moreover, if $N=2$, there can be no ramification of $g$ outside of $\Sigma_g^\infty$. If this is the case, the number of points of ${g}^* \mathbb{V}$ around which monodromy is of finite order is $d$ times the number of points in $\VV$ around which monodromy is of finite order. These numbers are obviously only equal if $d = 1$, so $g$ is an isomorphism in this case as well.
\end{proof}

This lemma does not apply when we have only one cusp, a situation which occurs whenever $n$ is prime \cite[Section 3.8]{fcmf}. We therefore need a second lemma to deal with the remaining cases, where $n \in \{5,7,11\}$.

\begin{lemma}\label{lemma:manyfinite}
Let $\mathbb{V}$ be a local system on a Zariski open subset of $\mathbb{P}^1$, such that one of the local monodromy matrices of $\mathbb{V}$ is of infinite order and there are at least three points in $\mathbb{P}^1$ at which the local monodromy matrices of $\mathbb{V}$ have finite order $\geq 2$.  Let $g \colon \mathbb{P}^1 \rightarrow \mathbb{P}^1$ be a morphism. Then ${g}^* \mathbb{V}$ has the same monodromy profile as $\mathbb{V}$ if and only if $g$ is an isomorphism.
\end{lemma}
\begin{proof}

For the same reason as in Lemma \ref{lemma:2infinite}, if $p_\infty \in \mathbb{P}^1$ is the point where the local monodromy of $\VV$ is of infinite order, $g^{-1}(p_\infty)$ is a single point and $g$ is ramified to order $d := \deg g$ at $g^{-1}(p_\infty)$. Therefore, the Riemann-Hurwitz formula gives
\[d-1 = \sum_{p \in \PP^1 \setminus g^{-1}(p_\infty)} (e_p -1).\]

Let $N$ be the number of points in $\mathbb{P}^1$ around which the local monodromy of $\mathbb{V}$ is of finite order $\geq 2$ and let $\Sigma_g^{< \infty}$ be the preimages of these points under $g$. Then the above equation gives
\[d - 1 \geq \sum_{p \in \Sigma_g^{< \infty}} (e_p -1) = dN - |\Sigma_g^{< \infty}|,\]
and thus
\[|\Sigma^{<\infty}_g|\geq d(N-1) + 1.\]

Moreover, if ${g}^* \mathbb{V}$ has the same monodromy profile as $\mathbb{V}$, there are at most $N$ points in $\Sigma_g^{<\infty}$ at which $g$ is unramified. If this bound is achieved, every other point in $\Sigma_g^{< \infty}$ has ramification degree at least $2$. Therefore,
\[N + 2(|\Sigma_g^{<\infty}| - N)  \leq \sum_{p \in \Sigma_g^{< \infty}} e_p = dN.\]
Thus 
\[|\Sigma_g^{<\infty}| \leq \dfrac{N(d+1)}{2}.\]
Combining the two inequalities involving $|\Sigma_g^{< \infty}|$, we obtain
\[d(N-1) + 1 \leq \dfrac{N(d+1)}{2}.\]
If $d \neq 1$, this can be rearranged to give $N \leq 2$. Therefore, if $N \geq 3$ as in the statement of the lemma, we must have $d = 1$, so $g$ is an isomorphism.
\end{proof}

This completes the proof of Proposition \ref{prop:modularity}.

\subsection{Singular fibres of $\calX_n$}

We wish to see which of the generalized functional invariant maps allowed by the restrictions of Sections \ref{sec:CYbound} and \ref{sec:furtherrestriction} can actually underlie non-rigid Calabi-Yau threefolds fibred by $M_n$-polarized K3 surfaces. As such threefolds are birational to the pull-backs of the fundamental families $\calX_n$ under the generalized functional invariant map (by Theorem \ref{genfuninv}), we begin by performing a detailed study of these fundamental families and their pull-backs under various covers.  

Of particular interest is the behaviour of these families, and their covers, in neighbourhoods of the orbifold points in $\calM_{M_n}$. The fibres over such points will, in general, be singular. We will examine each of these singular fibres in turn, beginning with the simplest kind: those lying over the $2$-orbifold points $\lambda_i$.

\begin{proposition} \label{prop:lambdafibres} Let $\Delta \subset \calM_{M_n}$ denote a small disc around one of the $2$-orbifold points $\lambda = \lambda_i$ listed in Table \ref{tab:orbipoints}. If $n \neq 7$, the threefold $\calX_n$ is smooth over $\Delta$ and its singular fibre over $\lambda = \lambda_i$ is a singular K3 surface containing an $A_1$ singularity. If we pull-back $\calX_n \to \calM_{M_n}$ by a $\mu$-fold cover $\Delta' \to \Delta$ ramified totally over $\lambda = \lambda_i$, the total space of the pulled-back threefold will contain an isolated $cA_{\mu-1}$ singularity.

In the case $n = 7$, the fibre over $\lambda = 27$ behaves as above. However, the fibre over $\lambda = -1$ contains an $A_2$ singularity and the threefold $\calX_n$ contains an isolated node \textup{(}$cA_1$\textup{)} singularity over this point. If we pull-back $\calX_n \to \calM_{M_n}$ by a $\mu$-fold cover $\Delta' \to \Delta$ ramified totally over $\lambda = -1$, the total space of the pulled-back threefold will contain an isolated $cA_{2\mu-1}$ singularity.
\end{proposition}

\begin{proof} The only difficult case here is when $n = 7$ and $\lambda = -1$. In this case, by Table \ref{tab:Mnsingularities}, $\bar{\calX}_7$ contains six curves of $cA_m$ ($m=1,2,3,4$) singularities over $\Delta$. When $\lambda = -1$, the singularity along one of the $cA_2$ curves (where $(x,y,z,t) = (0,1,-1,0)$) jumps to a $cA_3$. After blowing-up these six curves of singularities, we are thus left with an isolated $cA_1$ singularity over $\lambda = -1$ in $\calX_7$. The fibre of $\calX_7$ over $\lambda = -1$ is a singular K3, containing an $A_2$ singularity. Given this, the statement about the covers is easy. \end{proof}

\begin{remark} \label{singularityrem} In the case $n=7$, the threefold $\calX_n$ contains an isolated $cA_1$ singularity. Working \emph{analytically}, it is always possible to find a small crepant resolution of this singularity, to obtain a smooth threefold with singular fibre a nodal K3 surface as in the other cases. However, it may not be possible to perform such a resolution \emph{algebraically}.
\end{remark}

The next type of singular fibre that we turn our attention to is the fibres over the cusp $\lambda = \infty$. As we shall see, these fibres all admit a simple common description.

\begin{proposition} \label{prop:infinityfibres}  Let $\Delta \subset \calM_{M_n}$ denote a small disc around the cusp $\lambda = \infty$. The threefold $\calX_n$ is smooth over $\Delta$ and its singular fibre over $\lambda = \infty$ is a semistable K3 surface of Type III, containing $n+2$ components. If we pull-back $\calX_n \to \calM_{M_n}$ by a $\mu$-fold cover $\Delta' \to \Delta$ ramified totally over $\lambda = \infty$ and crepantly resolve any resulting singularities, the resulting threefold is again smooth and contains a semistable fibre of Type III with $n\mu^2 + 2$ components.
\end{proposition}
\begin{proof} Once we have proved the form of the fibres over $\Delta$, the form of the fibres over the $\mu$-fold cover $\Delta' \to \Delta$ follows immediately from \cite[Proposition 1.2]{bgd7}. We also note that the number of components appearing in the fibre of $\calX_n$ over $\lambda = \infty$ has previously been computed by Przyjalkowski in \cite[Corollary 35]{cyctlgmsft}, but he does not precisely describe these fibres; our computations corroborate his result. 

To find the precise forms of the fibres over $\Delta$, we proceed case-by-case, indexed by $n$. The resolutions presented here are relatively simple to compute, if tedious; for brevity we only sketch the details.

\begin{enumerate}[leftmargin=3em]
\item[(${M}_2$)]Follows from the proof of \cite[Proposition 2.5]{cytfmqk3s}.

\item[($M_3$)] By Table \ref{tab:Mnsingularities}, over $\Delta$ the threefold $\bar{\calX}_3$ contains nine curves of $cA_m$ ($m=1,2$) singularities, which form sections of the fibration. Away from these nine curves, $\bar{\calX}_3$ is nonsingular over $\Delta$. 
%On the fibre over $\lambda = \infty$, the singularities along these curves jump to $cA_{m+1}$'s.
After performing $m$ crepant blow-ups on each of these curves, we are left with nine isolated singularities of type $cA_1$ lying over $\lambda = \infty$. These may be resolved by a small projective resolution, after which the threefold $\calX_3$ becomes smooth over $\Delta$, with a semistable singular fibre of Type III consisting of five components arranged as a triangular prism.

\item[($M_4$)] By Table \ref{tab:Mnsingularities}, over $\Delta$ the threefold $\bar{\calX}_4$ contains twelve curves of $cA_1$ singularities, which form sections of the fibration. Away from these twelve curves, $\bar{\calX}_4$ is nonsingular over $\Delta$. 
%On the fibre over $\lambda = \infty$, the singularities along these curves jump to $cA_{2}$'s. 
After performing a crepant blow-up on each of these curves, we are left with twelve isolated $cA_1$ singularities over $\lambda = \infty$. These admit a small projective resolution, after which the threefold $\calX_4$ is smooth over $\Delta$, with a semistable singular fibre of Type III consisting of six components arranged in a cube.

\item[($M_5$)]By Table \ref{tab:Mnsingularities}, over $\Delta$ the threefold $\bar{\calX}_5$ contains three curves of $cA_1$ singularities, which we denote $C^{A_1}_i$ ($i = 1,2,3$), and three curves of $cD_4$ singularities, which we denote $C^{D_4}_j$ ($j=1,2,3$), all of which form sections of the fibration. The fibre of  $\bar{\calX}_5$ over $\lambda = \infty$ has four components, given by the four coordinate hyperplanes in $\mathbb{P}^3$, arranged as a tetrahedron, and the three curves $C^{D_4}_j$ pass through three of the vertices. Moreover, %the singularities of the curves $C^{A_1}_i$ jump to $cA_2$'s when $\lambda = \infty$, and 
the singularities along the curves $C^{D_4}_j$ jump to canonical threefold singularities when $\lambda = \infty$. 

These singularities may be crepantly resolved as follows. First perform a single embedded blow-up along each of the curves $C^{D_4}_j$. Each of these blow-ups is crepant and gives rise to two exceptional divisors: one lying over the curve $C^{D_4}_j$, and one lying over the point of intersection between $C^{D_4}_j$ and the fibre $\lambda = \infty$. After these blow-ups, we are left with twelve curves of generically $cA_1$ singularities: the three $C^{A_1}_i$ and three coming from each $C^{D_4}_j$. These may each be crepantly resolved by a single blow-up, leaving twelve isolated $cA_1$ singularities, which may be resolved by small projective blow-ups. 

After doing this, the threefold $\calX_5$ becomes smooth over $\Delta$ and its fibre over $\lambda = \infty$ has seven components: the strict transforms of the original four and three additional components coming from the vertices that lie on the $C^{D_4}_j$ ($j=1,2,3$). This fibre is semistable of Type III and has seven components arranged as a tetrahedron with three truncated vertices.

\item[($M_6$)] By Table \ref{tab:Mnsingularities}, over $\Delta$ the threefold $\bar{\calX}_6$ contains three curves of $cA_1$ singularities, denoted $C^{A_1}_1, C^{A_1}_2, C^{A_1}_3$, two curves of $cA_2$'s, denoted $C^{A_2}_1, C^{A_2}_1$, and two curves of $cA_3$'s, denoted $C^{A_3}_1,C^{A_3}_2$, which form sections of the fibration. There are also two further curves of $cA_1$ singularities lying in the fibre over $\lambda = \infty$, given by $D_1 := \{y = z = 0,\ \lambda = \infty\}$ and $D_2 := \{z = t = 0,\ \lambda = \infty\}$.

The fibre of $\bar{\calX}_6$ over $\lambda = \infty$ has four components arranged as a tetrahedron.   The curves $C^{A_3}_1$ and $D_2$ both pass through the point $x=z=t=0$ (which is a vertex of the tetrahedron), and the curves $C^{A_2}_1$, $D_1$ and $D_2$ all pass through the vertex $y=z=t=0$. Both of these vertices are canonical threefold singularities.

%There is also some isolated singularity behaviour on the fibre over $\lambda = \infty$: the singularities along $C^{A_1}_1$ and $C^{A_1}_2$ jump to $cA_2$'s, the singularity along $C^{A_2}_2$ jumps to an $cA_3$, the singularity at $C^{A_1}_3 \cap D_2$ is a $cA_3$, and the singularity at $C^{A_3}_2 \cap D_1$ is a $cA_5$. Finally, the singularity along $D_1$ jumps to an $cA_2$ at the vertex $x=y=z=0$, and there is an isolated threefold node ($cA_1$) at $(x,y,z,t) = (0,1,0,-1)$.
 
These singularities may be crepantly resolved by first performing $m$ blow-ups of each of the curves $C^{A_m}_i$, then blowing up the curves $D_j$ once each, and finally performing small projective resolutions on any isolated terminal singularities that remain. After doing this, the threefold $\calX_6$ becomes smooth over $\Delta$ and its fibre over $\lambda = \infty$ is semistable of Type III with eight components: the original four, plus one each from the curves $D_j$ ($j=1,2$) and one each from the vertices $x=z=t=0$ and $y=z=t=0$.

\item[($M_7$)] By Table \ref{tab:Mnsingularities}, over $\Delta$ the threefold $\bar{\calX}_7$ contains one curve of $cA_1$ singularities, denoted $C^{A_1}_1$, three curves of $cA_2$'s, denoted $C^{A_2}_1,C^{A_2}_2$, one curve of $cA_3$'s, denoted $C^{A_3}_1$, and one curve of $cA_4$'s, denoted $C^{A_4}_1$, all of which form sections of the fibration. 

To simplify the computation, first introduce a new variable $u := x + y + z + t$, then change coordinates $u \mapsto \lambda u$; this is an isomorphism on a general fibre and performs a birational modification to the fibre over $\lambda = \infty$,  but this birational modification does not affect the number of components appearing in a crepant resolution. The resulting family is given by
\begin{equation}\label{eq:M7modified} \lambda u^2yz + u(y+z+t)(z+t)^2 - (\lambda u - y - z- t)yzt = 0. \end{equation}
 
The fibre of this new family over $\lambda = \infty$ consists of four components arranged as a tetrahedron. In addition to the $C^{A_m}_i$, there are three curves of $cA_1$ singularities lying in this fibre, given by $D_1:=\{u=z=0\}$, $D_2 := \{u=t=0\}$, and $D_3 := \{u=y=0\}$. The curves $C^{A_4}_1$, $D_1$ and $D_2$ all pass through the vertex $u=z=t=0$, which is a canonical threefold singularity. %There is also some isolated singularity behaviour on the fibre over $\lambda = \infty$: the singularities along the curves $C^{A_2}_1$ and $C^{A_2}_2$ jump to $cA_3$'s. Moreover,  there is an isolated threefold node ($cA_1$) at $(y,z,t,u)=(1,0,-1,-1)$.

These singularities may be crepantly resolved as follows. First perform $m$ blow-ups along each of the curves $C^{A_m}_j$, then blow up each of the $D_j$ once. Finally, perform small projective resolutions on the isolated terminal singularities that remain. After this process the threefold $\calX_7$ becomes smooth over $\Delta$. Its fibre over $\lambda = \infty$ is semistable of Type III with nine components: the strict transforms of the original four, one component from each of the three $D_j$, and two additional components from the point $u=z=t=0$. 

\item[($M_8$)] By Table \ref{tab:Mnsingularities}, over $\Delta$ the threefold $\bar{\calX}_8$ contains six curves of $cA_1$ singularities, denoted $C^{A_1}_i$ ($i \in \{1,\ldots,6\}$), and three curves of $cA_2$'s, denoted $C^{A_2}_i$ ($i \in \{1,2,3\}$), which form sections of the fibration. There are also three additional curves $D_j$ ($j=1,2,3$) of $cA_1$ singularities lying in the fibre $\lambda = \infty$, given by $\{x = t = 0\}$, $\{y=t=0\}$ and $\{z=t=0\}$ respectively.

The fibre of $\bar{\calX}_8$ over $\lambda = \infty$ has four components, arranged as a tetrahedron, and the three $C^{A_2}_i$ pass through the three vertices $x=y=t=0$, $x=z=t=0$ and $x=y=t=0$, which are all canonical singularities. %There is also some isolated singularity behaviour on the fibre over $\lambda = \infty$: the singularities along the curves $C^{A_1}_i$ jump to $cA_2$'s for $i = 4,5,6$, and the singularities at the points $C^{A_1}_j \cap D_j$ ($j = 1,2,3$) are $cA_3$'s.

These singularities admit a crepant resolution as follows. First perform $m$ blow-ups of each of the curves $C^{A_m}_i$, then blow up each of the curves $D_j$ once. Finally, perform small projective resolutions on the isolated terminal singularities that remain. After doing this, the threefold $\calX_8$ becomes smooth over $\Delta$. Its fibre over $\lambda = \infty$ is semistable of Type III with ten components: the strict transforms of the original four, one component from each of the three $D_i$, and one component from each of the three vertices $C^{A_2}_i \cap D_i \cap D_j$.

\item[($M_9$)] By Table \ref{tab:Mnsingularities}, over $\Delta$ the threefold $\bar{\calX}_9$ contains three curves of $cA_1$ singularities, denoted $C^{A_1}_i$ ($i \in \{1,2,3\}$), three curves of $cA_2$ singularities, denoted $C^{A_2}_i$ ($i \in \{1,2,3\}$),  and one curve of $cD_4$ singularities, denoted $C^{D_4}$, all of which form sections of the fibration. There are also three additional curves $D_j$ ($j=1,2,3$) of $cA_1$ singularities in the fibre over $\lambda = \infty$, where $x = t = 0$, $y=t=0$ and $z=t=0$ respectively.

The fibre of $\bar{\calX}_8$ over $\lambda = \infty$ has four components, arranged as a tetrahedron. The three $C^{A_2}_i$ pass through the three vertices $x=y=t=0$, $x=z=t=0$ and $x=y=t=0$, and $C^{D_4}$ passes through the fourth vertex $x=y=z=0$. %There is also some isolated singularity behaviour on the fibre over $\lambda = \infty$: the singularities along the three curves $C^{A_1}_i$ ($i=1,2,3$) jump to $cA_3$'s, and all four vertices of the tetrahedron are canonical singularities with Milnor number $8$.

These singularities admit a crepant resolution as follows. First perform $m$ blow-ups of each of the curves $C^{A_m}_i$, then blow up the curves $D_j$ once each. Finally, perform an embedded blow-up along the curve $C^{D_4}$. This gives rise to an additional component over the vertex $\lambda=\infty$, $x=y=z=0$ and three disjoint curves of $cA_1$ singularities. These three curves may then be resolved by one further blow-up each. Finally, perform small projective resolutions on the isolated terminal singularities that remain. After this, the threefold $\calX_9$ becomes smooth over $\Delta$. Its fibre over $\lambda = \infty$ is semistable of Type III with eleven components: the strict transforms of the original four, one component from each of the three $D_j$, and one component from each of the four vertices of the tetrahedron.

\item[($M_{11}$)]By Table \ref{tab:Mnsingularities}, over $\Delta$ the threefold $\bar{\calX}_{11}$ contains  two curves of $cA_1$ singularities, denoted $C^{A_1}_1,C^{A_1}_2$, one curve of $cA_2$ singularities, denoted $C^{A_2}_1$,  and two curves of $cA_3$ singularities, denoted $C^{A_3}_1,C^{A_3}_2$, all of which form sections of the fibration. There are also four additional curves $D_j$ ($j=1,2,3,4$) of $cA_1$ singularities where $\lambda = \infty$ and $x = t = 0$, $y=t=0$, $x=z=0$, and $y=z=0$ respectively.

The fibre of $\bar{\calX}_{11}$ over $\lambda = \infty$ has four components, arranged as a tetrahedron. The curve $C^{A_2}_1$ passes through the vertex $x=y=t=0$, which is a canonical singularity, and the two curves $C^{A_3}_i$ ($i = 1,2$) pass through the vertices $x=z=t=0$ and $y=z=t=0$, each of which is a canonical singularity.

%There is also some isolated singularity behaviour on the fibre $\lambda = \infty$: the singularity along the curves $C^{A_1}_i$ ($i=1,2$) jumps to a $cA_3$, the singularities along $D_1$ and $D_2$ jump to $cA_2$'s at $(x,y,z,t) = (0,1,-1,0)$ and $(1,0,-1,0)$ respectively, and there is an isolated threefold node ($cA_1$) singularity at $(x,y,z,t) = (0,0,1,-1)$.

These singularities may be crepantly resolved as follows. First perform $m$ blow-ups along each of the curves $C^{A_m}_i$. After doing this, resolve the curves $D_j$ by blowing up once each. Finally, perform small projective resolutions on any isolated terminal singularities that remain. After this, the threefold $\calX_{11}$ becomes smooth over $\Delta$ and its fibre over $\lambda = \infty$ is semistable of Type III with thirteen components: the strict transforms of the original four, four additional components coming from the four $D_i$, one component from the vertex $x=y=t=0$ (where $C^{A_2}_1$, $D_1$ and $D_2$ meet), and two components from each of the two vertices $x=z=t = 0$ (where $C^{A_3}_1$, $D_1$ and $D_3$ meet) and $y = z= t = 0$ (where $C^{A_3}_2$, $D_2$ and $D_4$ meet).
\end{enumerate}

This completes the proof of Proposition \ref{prop:infinityfibres}.\end{proof}

The last, and most complex, singular fibres lie over the orbifold point $\lambda = 0$. In this case, Theorem \ref{thm:partitions} heavily restricts the degrees of the possible covers that we need to consider.

\begin{proposition}\label{prop:0fibres} Let $\Delta \subset \calM_{M_n}$ denote a small disc around the point $\lambda = 0$. The number of components of the singular fibre of $\calX_n$ over $\lambda = 0$ is given by the $\mu=1$ column of Table \ref{tab:0fibres}. 

After pulling-back $\calX_n \to \calM_{M_n}$ by a $\mu$-fold cover $\Delta' \to \Delta$ ramified totally over $\lambda = 0$, we may perform a birational modification \textup{(}that only affects the fibre over $\lambda = 0$\textup{)} so that the relative canonical bundle of the resulting threefold over $\Delta'$ is trivial and, unless marked otherwise \textup{(}see below\textup{)}, the resulting threefold is smooth. In this case, the number of components of the fibre over $\lambda = 0$ is given by the column corresponding to $\mu$ in Table \ref{tab:0fibres}.

Moreover, in Table \ref{tab:0fibres} 
\begin{itemize}
\item in cases marked with an asterisk \textup{(}$\mbox{}^*$\textup{)}, the corresponding fibre is semistable of Type III or, if there is only one component, is a smooth K3 surface; and
\item there are several possibilities for fibres having only one component. A fibre denoted $1^*$ \textup{(}resp. $1$\textup{)} is a smooth \textup{(}resp. nodal\textup{)} K3 surface in a smooth total space. A fibre denoted $1^{*\dagger}$ \textup{(}resp. $1^{\dagger}$\textup{)} is a singular K3 surface in a singular total space, that admits a small analytic resolution to a smooth \textup{(}resp. nodal\textup{)} K3 surface in a smooth total space. Note that in all cases marked with $\mbox{}^{\dagger}$ it may not be possible to perform such resolutions algebraically, see Remark \ref{singularityrem}.
\end{itemize}
\end{proposition}

\begin{table}
\begin{tabular}{|c|c|c|c|c|c|c|c|c|}
\hline
 & \multicolumn{8}{c|}{Degree $\mu$ of ramified cover} \\
$n$ & $1$ & $2$ & $3$ & $4$ & $5$& $6$ & $7$ & $8$ \\
\hline
$2$ & $31$ & $11$ & $4$  & $1^*$ & $31$ & $11$ & $4$ & $1^*$ \\
$3$ & $21$ & $6$ & $1$ & $21$ & $6$ & $1^{*\dagger}$ & $-$ & $-$ \\
$4$ & $15$ & $3^*$ & $18$ & $6^*$ & $-$ & $-$ & $-$ & $-$\\
$5$ & $11$ & $1^*$ & $11$ & $1^*$ & $-$  & $-$ &$-$ & $-$ \\
$6$ & $8^*$ & $26^*$ & $-$ & $-$ & $-$ & $-$ &$-$ & $-$\\
$7$ & $6$ & $15$ & $1^{*\dagger}$ & $-$ & $-$ & $-$ &$-$ & $-$\\ 
$8$ & $4^*$ & $10^*$ & $-$ & $-$ & $-$ & $-$ &$-$ & $-$\\
$9$ & $3^*$ & $6^*$ & $-$ & $-$ & $-$ & $-$ &$-$ & $-$\\
$11$ & $1^{\dagger}$ & $1^{*\dagger}$ & $-$ & $-$ & $-$ & $-$ &$-$ & $-$\\
\hline
\end{tabular}
\caption{Number of components in the fibre over $\lambda = 0$ in $\calX_n$ and its ramified covers.}
\label{tab:0fibres}
\end{table}

\begin{proof} As in the proof of Proposition \ref{prop:infinityfibres}, we proceed by analyzing each case in turn, indexed by $n$. Note that in some of these cases, including all cases with $\mu =1$, the number of components in the singular fibre has previously been computed by Przyjalkowski \cite{wlgmsft}; in these cases our computations reproduce his result. As before, for brevity we only sketch the details of the computations.
\begin{enumerate}[leftmargin=3em]
\item[($M_2$)] The cases $\mu = 1,2,4,8$ were computed in the proof of \cite[Proposition 2.5]{cytfmqk3s}. 

In the case $\mu = 3$, after pulling-back $\bar{\calX}_2$ by a triple cover the threefold becomes non-normal along the locus $\lambda = w = 0$. Blow this locus up in the fourfold ambient space. After blowing up, the strict transform of our threefold is smooth away from the six curves of $cA_3$ singularities which form sections of the fibration, and its fibre over $\lambda = 0$ has five components: the strict transform of the original along with four new components. Of these, the component which is a strict transform of the original fibre is exceptional and may be contracted, leaving a threefold with trivial relative canonical bundle. The singular fibre of this threefold over $\lambda = 0$ has four components, arranged as a cone over four lines in $\PP^2$. 

In the cases $\mu = 5,6,7$, after making a change of coordinates $t \mapsto \lambda t$ and normalizing, we obtain a singular fibre with the same configuration of singularities as in the cases $\mu = 1,2,3$, respectively. These may be resolved in the same way, to give singular fibres with respectively $31$, $11$ and $4$ components.

\item[($M_3$)] In $\bar{\calX}_3$, the fibre $\lambda = 0$ has two components: a component $S := \{\lambda = s = 0\} \cong \PP^2$, which has multiplicity $2$, and a component $Z:= \{\lambda = z = 0\} \cong \PP^1 \times \PP^1$, which has multiplicity $3$. Moreover, over $\Delta$, the threefold $\bar{\calX}_3$ has $17$ curves of singularities: along with the six curves $C^{A_1}_i$ ($i \in \{1,\ldots,6\}$) of $cA_1$ and three curves $C^{A_2}_i$ ($i\in\{1,2,3\}$) of $cA_2$ singularities coming from Table \ref{tab:Mnsingularities}, which form sections of the fibration,  there are also three curves $D_i$ of $cA_1$ singularities in the component $S$ and five curves $E_i$ of $cA_2$ singularities in the component $Z$. The method of \cite[Section 1]{bgd7} may be used to simultaneously resolve all curves of $cA_1$'s, followed by all curves of $cA_2$'s, giving rise to a singular fibre with $21$ components: the strict transforms of $S$ and $Z$, one exceptional component from each of the three $D_i$, two exceptional components from each of the five $E_i$, and one exceptional component from each of the six points $C^{A_2}_i \cap E_j \cap E_k$.

In the case $\mu=2$, suppose that we pull-back $\bar{\calX}_3$ by a double cover $\Delta' \to \Delta$ ramified over $\lambda = 0$. The resulting threefold is non-normal along its central fibre. After performing a change of coordinates $s \mapsto \lambda s$, which has the effect of normalizing the pull-back of $S$ and contracting the pull-back of $Z$, we obtain a threefold with irreducible central fibre that contains ten curves of singularities over $\Delta'$: the pull-backs of the nine $C^{A_m}_i$ (note that the six $C^{A_2}_i$'s are no longer disjoint, but intersect in three pairs) and the curve $D := \{\lambda = r = z = 0\}$, which is a curve of $cA_2$ singularities. Once again, the configuration of curves of $cA_2$ singularities may be simultaneously resolved by the method of \cite[Section 1]{bgd7}, leaving three disjoint curves $C^{A_1}_i$ of $cA_1$'s, which may be resolved by one blow-up each. The result is a singular fibre with six components: the strict transform of $S$, two from the curve $D$, and one each from the three points $C^{A_2}_i \cap C^{A_2}_j \cap D$.

In the case $\mu=3$, suppose that we pull-back $\bar{\calX_3}$ by a triple cover $\Delta' \to \Delta$ ramified over $\lambda = 0$. The resulting threefold is non-normal along its central fibre. After performing a change of coordinates $z \mapsto \lambda z$, which has the effect of normalizing the pull-back of $Z$ and contracting the pull-back of $S$, we find that the resulting threefold is smooth away from the nine curves $C^{A_j}_i$ (note that the three $C^{A_1}_i$'s are no longer disjoint, and all intersect at the point $\lambda = s = x =y = 0$). Blow-up the six curves $C^{A_2}_i$ twice each, then simultaneously resolve the three curves $C^{A_1}_i$ by the method of \cite[Section 1]{bgd7} to obtain a smooth threefold whose singular fibre over $\lambda = 0$ is a nodal K3. 

Given this it is easy to see that, in the case $\mu=6$, the pull-back of $\calX_3$ by a $6$-fold cover $\Delta' \to \Delta$ is a singular threefold, with an isolated node ($cA_1$) singularity.

In the case $\mu=4$, suppose that we pull-back $\calX_3$ by a $4$-fold cover $\Delta' \to \Delta$ ramified over $\lambda = 0$. After making a change of coordinates $z \mapsto \lambda z$, we obtain a central fibre containing three curves of $cA_1$ singularities and five curves of $cA_2$ singularities, arranged in the same configuration as the unramified case. This may be crepantly resolved in the same way as the unramified case, to give a fibre with $21$ components. 

Finally, in the case $\mu=5$, suppose that we pull-back $\bar{\calX}_3$ by a $5$-fold cover $\Delta' \to \Delta$ ramified over $\lambda = 0$. After performing a change of coordinates $s \mapsto \lambda s$, $z \mapsto \lambda z$, we obtain a threefold with irreducible central fibre that contains ten curves of singularities over $\Delta'$: the pull-backs of the nine $C^{A_m}_i$ (as in the double covering case, the six $C^{A_2}_i$'s are no longer disjoint, but intersect in three pairs, and the three $C^{A_1}_i$'s intersect transversely in a triple point) and the curve $D := \{\lambda = r = z = 0\}$, which is a curve of $cA_2$ singularities. As in the double covering case, the configuration of curves of $cA_2$ singularities may be simultaneously resolved by the method of \cite[Section 1]{bgd7}, leaving the three curves $C^{A_1}_i$ of $cA_1$'s, which may be simultaneously resolved by the same method. The result is a singular fibre with six components: the original component, two from the curve $D$, and one each from the three points $C^{A_2}_i \cap C^{A_2}_j \cap D$.

\item[($M_4$)]  In $\bar{\calX}_4$ the fibre over $\lambda = 0$ has three components $S_1,S_2,S_3$, given by $S_i := \{\lambda = s_i = 0\} \cong \PP^1 \times \PP^1$, each of which has multiplicity $2$. Moreover, over $\Delta$ the threefold $\bar{\calX}_4$ has $24$ curves of singularities: along with the twelve curves $C_i$ ($i = 1,\ldots,12$) of $cA_1$ singularities from Table \ref{tab:Mnsingularities}, which form sections of the fibration, there are also four curves $D_{ij}$ ($i \in \{1,2,3\}$, $j \in \{1,\ldots,4\}$) of $cA_1$ singularities in each component $S_i$, corresponding to the intersections of $S_i$ with the hyperplanes $\{r_k = 0\}$ and $\{r_k = s_k\}$ for $k \neq i$. This configuration can be resolved by the method of \cite[Section 1]{bgd7}, giving rise to a singular fibre with $15$ components: the strict transforms of the three $S_i$ along with one exceptional component from each of the twelve $D_{ij}$.

In the case $\mu=2$, suppose that we pull-back $\bar{\calX}_4$ by a double cover $\Delta' \to \Delta$ ramified over $\lambda = 0$. The resulting threefold is non-normal. The normalization is singular only along the pull-backs of the curves $C_i$ and its central fibre consists of three components, each of which is a double cover of one of the $S_i$ ramified over the four curves $D_{ij}$ (each component thus obtained contains four $A_1$ singularities, corresponding to the intersections $D_{ij} \cap D_{ik}$. These coincide with the intersection of $S_i$ with the curves $C_l$, which are blown up in the next step). After blowing up the $C_i$, we obtain a smooth threefold whose central fibre is semistable of Type III, with three components meeting along three double curves and two triple points. 

Given this, the form of the fibre over the pull-back by a further double cover (corresponding to $\mu=4$) then follows easily from \cite[Proposition 1.2]{bgd7}. 

Finally, in the case $\mu=3$, suppose that we pull-back $\bar{\calX}_4$ by a triple cover $\Delta' \to \Delta$ ramified over $\lambda = 0$. The resulting threefold is again non-normal. After normalizing, the resulting fibre has $cA_1$ singularities along the twelve $C_i$, the twelve $D_{ij}$, and along the three intersections $S_i \cap S_j$. Crepantly resolving these singularities using the method of \cite[Section 1]{bgd7}, we obtain a central fibre with $18$ components: the original three, plus one from each of the twelve curves $D_{ij}$ and one from each of the three curves $S_i \cap S_j$.

\item[($M_5$)]  The fibre of $\calX_5$ over $\lambda = 0$ has been previously studied by Iliev, Katzarkov and Przyjalkowski \cite[Section 5.2]{dscnr}. They show that after a number of crepant blow-ups the threefold $\calX_5$ becomes smooth over $\Delta$ and its fibre over $\lambda = 0$ has $11$ components.

In the case $\mu=2$, suppose that we pull-back $\bar{\calX}_5$ by a double cover $\Delta' \to \Delta$ ramified over $\lambda = 0$. The resulting threefold is non-normal. After normalizing, the only singularities remaining are the six curves of $cA_1$ and $cD_4$ singularities coming from Table \ref{tab:Mnsingularities}. After blowing up these curves, the resulting threefold is smooth over $\Delta'$ with central fibre a smooth K3 surface. A further double cover preserves this smoothness, giving the case $\mu=4$. 

Finally, in the case $\mu=3$, suppose we pull-back $\bar{\calX}_5$ by a triple cover $\Delta' \to \Delta$ ramified over $\lambda = 0$. The resulting threefold is non-normal along its central fibre $\lambda = 0$. We may normalize by blowing up this fibre in the fourfold total space. The resulting threefold is normal and its fibre over $\lambda = 0$ contains seven curves of $cA_1$ singularities, arranged in the same configuration as the unramified case. This may be crepantly resolved by the method of \cite[Section 5.2]{dscnr}, to give a fibre with $11$ components.

\item[($M_6$)]  Note there is a natural involution of the family $\bar{\calX}_6$, given by 
\[ (x,y,z,t,\lambda) \longmapsto (y+z,-z-t,x+z+t,-x-y-z-t,\tfrac{1}{\lambda}).\]
This involution exchanges the fibres over $\lambda = 0$ and $\lambda = \infty$, which must therefore be isomorphic. Given this, the properties for the fibre over $\lambda = 0$ follow immediately from the properties of the fibre over $\lambda = \infty$, as computed in Proposition \ref{prop:infinityfibres}.

\item[($M_7$)]  The fibre of $\calX_7$ over $\lambda = 0$ is the most difficult. We once again use the simplified form \eqref{eq:M7modified}, but with an additional variable substitution $w := z+t$. The equation for the family becomes
\[\lambda u^2yz + uw^2(y + w) - yz(\lambda u - y - w)(w-z) = 0.\]
The fibre of this family over $\lambda = 0$ has two components, given by $\{y+w = 0\}$ and $\{(uw^2+yz(w-z) = 0\}$, which meet along the two curves $D_1 := \{y=w=0\}$ and $D_2 := \{y+w = uw-z(w-z) = 0\}$. Over $\Delta$ the threefold defined by this equation contains eight curves of singularities: along with the six curves $C^{A_m}_i$ ($m \in \{1,\ldots,4\}$, $i \in \{1,2,3\}$) of $cA_m$ singularities coming from Table \ref{tab:Mnsingularities}, which form sections of the fibration, there is an additional curve of $cA_2$ singularities along $D_1$ and a curve $D_3 := \{ w=z=0\}$ of $cA_1$ singularities.

This fibre admits a crepant resolution as follows. First blow up the curve $D_1$ twice and the curve $D_3$ once. After this, perform $m$ blow-ups of each of the curves $C^{A_m}_i$. Finally, any remaining isolated terminal singularities may be resolved by small projective resolutions. After this process, $\calX_7$ becomes smooth over $\Delta$, and its fibre over $\lambda = 0$ has six components: the strict transforms of the original two, two from $D_1$, one from $D_3$, and one from the point $y=z=w=0$ where $D_1$, $D_3$ and $C^{A_2}_1$ meet.

In the case $\mu=2$, suppose that we pull-back by a double cover. The fibre described above contains six components meeting along nine double curves. Under the cover, these double curves become curves of $cA_1$ singularities, which may be resolved by blowing up once each. The resulting fibre has $15$ components.

Finally, in the case $\mu=3$, suppose that we pull-back by a triple cover ramified over $\lambda = 0$. Perform another change of coordinates $u \mapsto \frac{u}{\lambda}$, $w \mapsto \lambda w$ and $y \mapsto \lambda y$. The resulting family is
\[u^2yz + uw^2(y+z) - yz(\lambda u - y - w)(\lambda w - z) = 0.\]
This family has an irreducible fibre over $\lambda = 0$ and the threefold total space is smooth away from the six curves $C^{A_m}_i$, which form sections of the fibration. However, there is some isolated singularity behaviour where the curves $C^{A_m}_i$ meet the fibre $\lambda = 0$: the curves $C_1^{A_1}$ and $C_1^{A_2}$ collide in a $cA_4$ singularity, the curves $C_2^{A_2}$ and $C_1^{A_3}$ collide in a $cA_6$ singularity, and the singularity along the curve $C_1^{A_4}$ jumps to a $cD_5$. After blowing up these three singularities, we are left with three isolated nodes ($cA_1$'s) in the threefold total space, corresponding to the three points where the isolated singularity behaviour occurs. The fibre over $\lambda = 0$ in this threefold is a singular K3 surface with three $A_1$ singularities, one for each threefold node.

\item[($M_8$)]  The fibre of $\calX_8$ over $\lambda = 0$ has been previously studied by Przyjalkowski \cite[Example 29]{wlgmsft}. He shows that after a number of crepant blow-ups the threefold $\calX_8$ becomes smooth over $\Delta$ and its fibre over $\lambda = 0$ is semistable with $4$ components, arranged as a tetrahedron. 

Given this, in the case $\mu=2$ the form of the fibre over the double cover $\Delta' \to \Delta$ follows immediately from \cite[Proposition 1.2]{bgd7}.

\item[($M_9$)]  The fibre of $\calX_9$ over $\lambda = 0$ has also been previously studied by Przyjalkowski \cite[Example 30]{wlgmsft}, who shows that after a number of crepant blow-ups the threefold $\calX_9$ becomes smooth over $\Delta$ and its fibre over $\lambda = 0$ is semistable with $3$ components. 

As before, in the case $\mu=2$, the form of the fibre over the double cover $\Delta' \to \Delta$ follows immediately from \cite[Proposition 1.2]{bgd7}.

\item[($M_{11}$)]  The fibre of $\bar{\calX}_{11}$ over $\lambda = 0$ is irreducible, and the only singularities of $\bar{\calX}_{11}$ over $\Delta$ are the five curves of $cA_m$ ($m=1,2,3$) singularities arising from Table \ref{tab:Mnsingularities} and one isolated node ($cA_1$) at the point $(x,y,z,t,\lambda) = (1,1,-1,0,0)$. After blowing up the five curves of $cA_m$ singularities, $\calX_{11}$ is left with a single isolated node ($cA_1$). Its fibre over $\lambda = 0$ is a singular K3 surface with an $A_2$ singularity.  

In the case $\mu=2$, if we pull-back $\calX_{11}$ by an double cover ramified over $\lambda = 0$, the fibre over $\lambda = 0$ is unchanged, but the resulting threefold will contain an isolated $cA_{3}$ singularity over $\lambda = 0$.
\end{enumerate}
This completes the proof of Proposition \ref{prop:0fibres}.\end{proof}

\section{Constructing Calabi-Yau threefolds} \label{sec:constCY3}

The aim of this final section is to use the results of the previous sections to explicitly classify nonrigid Calabi-Yau threefolds fibred non-isotrivially by $M_n$-polarized K3 surfaces, for $n \neq 1$. By Theorem \ref{genfuninv}, we know that any such Calabi-Yau threefold is birational to the pull-back $\bar{\calX}_{n,g}$ of one of the families $\bar{\calX}_n \to \calM_{M_n}$ under the generalized functional invariant map $g \colon \PP^1 \to \calM_{M_n}$, and Corollary \ref{cor:nonrigid} tells us that $n \in \{2,3,4,5,6,7,8,9,11\}$.

We are thus reduced to classifying the possible generalized functional invariant maps that give rise to Calabi-Yau threefolds in each case. We may use Theorem \ref{thm:partitions} to significantly reduce the number of cases that we need to consider. We thus obtain the following theorem, which should be thought of as a converse to Theorem \ref{thm:partitions}.

The generalized functional invariant maps $g$ appearing in this theorem are described using the notation introduced in Section \ref{sec:furtherrestriction}. Recall that $d$ denotes the degree of $g$ (note that $d \geq 2$, by Theorem \ref{thm:partitions}) and $[x_1,\ldots,x_k]$, $[y_1,\ldots,y_l]$, $[z_{1,1},\ldots,z_{m_1,1}],\ldots,[z_{1,1},\ldots,z_{m_q,q}]$ denote the partitions of $d$ that encode the ramification profiles of $g$ over $\lambda = \infty$ (the cusp), $\lambda = 0$ (the $a$-orbifold point) and $\lambda = \lambda_1,\ldots,\lambda_q$ (the $2$-orbifold points) respectively. Finally, $r$ denotes the degree of ramification of $g$ away from $\lambda \in \{0,\infty,\lambda_1,\ldots,\lambda_q\}$.

\begin{theorem} \label{thm:CY3} With notation as  in Section \ref{sec:furtherrestriction}, suppose $n \in \{2,3,4,5,6,7,8,9,11\}$,
\[k+l+m_1+\cdots + m_q-qd-r-2=0\]
and $[y_1,\ldots,y_l]$ is one of the partitions of $d$ listed in Table \ref{tab:partitions}. Then $\bar{\calX}_{n,g}$ is birational to a Calabi-Yau threefold $\calX_{n,g}$ with at worst $\QQ$-factorial terminal singularities. 

Moreover, if 
\begin{itemize}
\item $n \neq 7, 11$, and
\item $g$ is unramified over $\lambda \in \{\lambda_1,\ldots,\lambda_q\}$, and
\item we are not in the case $n = 3$, $l=1$, $[y_1] = [6]$,
\end{itemize}
then $\calX_{n,g}$ is smooth.
\end{theorem}

\begin{proof} The singularities in $\bar{\calX}_{n,g}$ away from the fibres over the points in the preimage $g^{-1}(\{0,\infty,\lambda_1,\ldots,\lambda_q\})$ are given in Table \ref{tab:Mnsingularities}; they all lie along smooth curves which form sections of the fibration, so may be crepantly resolved. In neighbourhoods of the fibres over $g^{-1}(\{0,\infty,\lambda_1,\ldots,\lambda_q\})$, crepant birational tranformations resolving the singularities in $\bar{\calX}_{n,g}$ are computed in Propositions \ref{prop:lambdafibres}, \ref{prop:infinityfibres} and \ref{prop:0fibres}. Using these explicit birational transformations, it is a straightforward calculation to show that $\calX_{n,g}$ has trivial canonical sheaf. Finally, vanishing of $H^1(\calX_{n,g},\calO_{\calX_{n,g}})$ follows by the method used to prove \cite[Proposition 2.7]{cytfmqk3s}. 
\end{proof}

\begin{remark} \label{rem:missing} This theorem should be thought of as a generalization of the results of \cite[Section 2]{cytfmqk3s}. Note that, even in the $n = 2$ case, the main result \cite[Corollary 2.8]{cytfmqk3s} of that section is somewhat weaker than the result above, as it only classifies the cases where the threefolds $\bar{\calX}_{2,g}$ have Calabi-Yau \emph{resolutions}. In contrast, the result above classifies all cases where $\bar{\calX}_{2,g}$ is \emph{birational} to a Calabi-Yau threefold, which is a weaker condition. In particular, here we acquire additional cases $[y_1,\ldots,y_l] = [1,3]$, $[2,3]$, $[3,3]$, $[3,4]$, $[5]$, $[6]$, and $[7]$ (also see Remark \ref{rem:delta}).
\end{remark}

Using the methods from \cite{cytfmqk3s}, it is now a simple matter to compute the Hodge numbers of these threefolds.

\begin{proposition} \label{prop:h11} Let $\calX_{n,g}$ be a Calabi-Yau threefold as in Theorem \ref{thm:CY3}. Suppose that $\calX_{n,g}$ is smooth. Then
\[h^{1,1}(\calX_{n,g}) = 20 + \sum_{i=1}^{k}\left(nx_i^2+1\right) + \sum_{j=1}^l \left(c(n,y_j) - 1\right)\]
where $[x_1,\ldots,x_k]$ and $[y_1,\ldots,y_l]$ are the partitions of $d$ encoding the ramification profiles of $g$ over $\lambda = \infty$ and $\lambda = 0$ respectively, and $c(n,y_j)$ is the integer in row $n$ and column $y_j$ of Table \ref{tab:0fibres}.
\end{proposition}
\begin{proof} This is a straightforward computation using \cite[Lemma 3.2]{cytfmqk3s}.\end{proof}

\begin{proposition} \label{prop:h21}  Let $\calX_{n,g}$ be a Calabi-Yau threefold as in Theorem \ref{thm:CY3}. Suppose that $\calX_{n,g}$ is smooth. Then we have
\[h^{2,1}(\calX_{n,g}) = k+l -2 + \tfrac{1}{2}(m_{\mathrm{odd}} - qd) + \delta,\]
 where 
\begin{itemize}
\item $k$ and $l$ denotes the number of ramification points of $g$ over the points $\lambda = \infty$ and $\lambda = 0$ respectively,
\item $m_{\mathrm{odd}}$ denotes the number of ramification points of odd order lying over the $2$-orbifold points in the set $\{\lambda_1,\ldots,\lambda_q\}$, 
\item $d$ is the degree of $g$,
\item $\delta$ is a correction term which equals $2$ if $n=2$ and $[y_1,\ldots,y_l]$ is $[3,3]$ or $[7]$; equals $1$ if 
\begin{itemize}
\item $n = 2$ and $[y_1,\ldots,y_l]$ is one of $[1,3]$, $[2,3]$, $[3,4]$, $[5]$, $[6]$, $[8]$, or
\item $n = 3$ and $[y_1,\ldots,y_l] = [5]$, or
\item $n=5$ and  $[y_1,\ldots,y_l]$ is $[3]$ or $[4]$;
\end{itemize}
and equals $0$ otherwise \textup{(}see Remark \ref{rem:delta}\textup{)}.
\end{itemize}
\end{proposition}

\begin{proof} The proof of this proposition follows by the same method as was used to prove \cite[Corollary 3.9]{cytfmqk3s}. We begin by setting up some notation: let $\pi_{n,g}\colon \calX_{n,g} \to B \cong \PP^1$ denote the K3 fibration on $\calX_{n,g}$, let $j\colon U \hookrightarrow B$ denote the open set over which $\calX_{n,g}$ is smooth and the fibres of $\calX_{n,g}$ are smooth K3 surfaces, and let $\pi_{n,g}^U$ denote the restriction of $\pi_{n,g}$ to $U$. 

Note that, since $\calX_{n,g}$ is a smooth Calabi-Yau threefold, we have $h^{2,1}(\calX_{n,g}) = \frac{1}{2}b_3(\calX_{n,g}) -1$, so it suffices to calculate $b_3(\calX_{n,g})$. It follows from \cite[Lemma 3.4]{cytfmqk3s} that $b_3(\calX_{n,g}) = h^1(B,j_*R^2(\pi_{n,g}^U)_*\CC)$ and, using Equation \eqref{eq:splitting}, we have 
\[h^1(B,j_*R^2(\pi_{n,g}^U)_*\CC) = h^1(B,j_*\calT(\calX_{n,g}^U)).\] 
Now, by Proposition \ref{prop:pullback}, we have $\calT(\calX_{n,g}^U) = g^*\VV_n^+$, so we obtain $b_3(\calX_{n,g}) = h^1(B,j_*g^*\VV_n^+)$,
which may be computed using Equation \eqref{eq:poincare} and the explicit monodromy matrices computed in Section \ref{sec:localmonodromy}.

We illustrate this computation for $n=4$, the other cases are analogous. In this case:
\begin{itemize}
\item if $g$ ramifies to order $y$ at a preimage $p$ of $0$, then $R(p) = 4 - \mathrm{hcf}(y,2)$,
\item if $g$ ramifies to order $z$ at a preimage $p$ of $\lambda_1 = 64$, then $R(p) = 2 - \mathrm{hcf}(z,2)$, and
\item at any preimage $p$ of $\infty$, we have $R(p) = 2$.
\end{itemize}
Thus Equation \eqref{eq:poincare} gives us
\begin{align*} b_3(\calX_{4,g}) &= h^1(B,j_*g^*\VV_4^+) \\
 &= \sum_{i=1}^l(4 - \mathrm{hcf}(y_i,2)) + \sum_{j=1}^{m_1}(2 - \mathrm{hcf}(z_{j,1},2)) + 2k - 6.\\
&= 2k + 2l + m_{\mathrm{odd}} - 6 + \sum_{i=1}^l (2 - \mathrm{hcf}(y_i,2)). \end{align*}
Note that the value $\sum_{i=1}^l (2 - \mathrm{hcf}(y_i,2))$ is the number of ramification points of odd order lying over the point $\lambda = 0$. Examining the possible ramification profiles over this point, as given in Table \ref{tab:partitions}, we see that when $n=4$ this value is always equal to $4-d$. So we have
\[h^1(B,j_*g^*\VV_4^+) = 2k + 2l -2 + (m_{\mathrm{odd}} - d).\]
The identity $h^{2,1}(\calX_{4,g}) = \frac{1}{2}b_3(\calX_{4,g}) -1$ then gives
\[h^{2,1}(\calX_{4,g}) = k+l - 2 + \tfrac{1}{2} (m_{\mathrm{odd}} - d),\]
as required.
 \end{proof}

\begin{remark} \label{rem:delta} The correction term $\delta$ deserves a brief explanation here. Indeed,
\begin{itemize}
\item when $n=2$, the cases $[y_1,\ldots,y_l] = [a]$, for $a \in \{5,6,7,8\}$, are smooth degenerations of the cases $[y_1,\ldots,y_l] = [4,a-4]$,
\item when $n=3$, the case $[y_1,\ldots,y_l] = [5]$ is a smooth degeneration of the case $[y_1,\ldots,y_l] = [2,3]$, and
\item when $n=5$, the cases $[y_1,\ldots,y_l] = [a]$, for $a \in \{3,4\}$, are smooth degenerations of the cases $[y_1,\ldots,y_l] = [2,a-2]$.
\end{itemize}
Moreover, when $n=2$, the case with ramification profile $[x_1,\ldots,x_k]$ over $\lambda = \infty$ and  $[3,a]$ over $\lambda = 0$, for some $a \in \{1,2,3,4\}$,  is a degeneration of the case with ramification profile $[1,x_1,\ldots,x_k]$ over $\lambda = \infty$ and $[4,a]$ over $\lambda = 0$.

The values of  $h^{2,1}(\calX_{n,g})$ must be invariant under these degenerations, and the correction term $\delta$ ensures that this holds. In this way, $\delta$ may be thought of as a correction for ``non-genericity''. It corrects for the fact that, for certain choices of $n,g$ as above, the Calabi-Yau threefold $\calX_{n,g}$ is not generic in its deformation class; it may be thought of as measuring the codimension of the  family $\calX_{n,g}$ in moduli.

Note that this also means that, whilst \cite[Corollary 2.8]{cytfmqk3s} did not capture all of the threefolds with $n=2$ from Theorem \ref{thm:CY3} (see Remark \ref{rem:missing}), it did capture the generic member of every deformation class (i.e. all cases with $\delta = 0$).
\end{remark}

\begin{remark} This formula for $h^{2,1}(\calX_{n,g})$ can be simplified significantly if we make the genericity assumptions that $\delta = 0$ (so $\calX_{n,g}$ is generic in its deformation class) and $g$ is unramified over $\lambda \in \{\lambda_1,\ldots,\lambda_q\}$. In this case, we have $m_{\mathrm{odd}} = m_1+\cdots + m_q = qd$, so it follows from Theorem \ref{thm:CY3} that $h^{2,1}(\calX_{n,g}) = r$, the degree of ramification away from $\lambda \in \{0,\infty,\lambda_1,\ldots,\lambda_q\}$.

In this setting, an argument analogous to that used to prove \cite[Proposition 4.1]{cytfmqk3s} shows that all small deformations of $\calX_{n,g}$ are induced by deformations of $g$ that preserve the ramification profiles over $\lambda \in \{0,\infty\}$. One may use this to compare the moduli spaces of the Calabi-Yau threefolds $\calX_{n,g}$ to the Hurwitz spaces of the maps $g$.
\end{remark}

\begin{remark}\label{rem:mirrors2} It is interesting to look at the mirror symmetric interpretation of these results. The philosophy of \cite{mstdfcym} states that a smooth Calabi-Yau $\calX_{n,g}$ with $l=2$ (so there are $2$ ramification points over $\lambda = 0$) should split into a pair of LG-models, each corresponding to a quasi-Fano variety. These quasi-Fano varieties may then be glued to give a \emph{Tyurin degeneration}, which smooths to give a Calabi-Yau variety mirror to $\calX_{n,g}$. Further evidence for this conjecture is given in \cite{mpcytmpqft}.

We can say more: if $\calX_{n,g}$ is a smooth Calabi-Yau with $\delta = 0$ (so $\calX_{n,g}$ is generic in its deformation class) and $l=2$, and $[y_1,y_2]$ is the ramification profile over $\lambda = 0$, then the pairs $(n,y_i)$ should equal (hypersurface degree, index) for the corresponding quasi-Fano varieties. Indeed, one may immediately observe that the pairs $(n,y_i)$ satisfying these assumptions are 
\[\begin{array}{c}(2,1), (2,2), (2,4), (3,1), (3,2), (3,3), (4,1), (4,2),\\ (5,1), (5,2), (6,1), (7,1), (8,1), (9,1), (11,1),\end{array}\]
which give all fifteen possible (hypersurface degree, index) pairs for smooth Fano varieties with Picard rank $1$ and hypersurface degree $\geq 2$.

\cite{mstdfcym} does not provide a mirror interpretation for the Calabi-Yau varieties $\calX_{n,g}$ with $l=1$. The authors plan to address this in future work.
\end{remark}

\bibliography{Publications,Preprints}
\bibliographystyle{amsalpha}
\end{document}